\documentclass[11pt]{article}
\usepackage{amsmath,amsthm,amssymb,amscd}
\usepackage{ amssymb }
\usepackage{tikz}
\usepackage[all]{xy}
\usepackage[german,frenchb,english]{babel}
\usepackage{amsmath}
\usepackage{amssymb}
\usepackage{geometry}
\usepackage[backref]{hyperref}
\usepackage{ textcomp }

\DeclareMathOperator{\sm}{sm}

\DeclareMathOperator{\coker}{coker}

\DeclareMathOperator{\dual}{\vee}

\DeclareMathOperator{\Hilb}{Hilb}

\DeclareMathOperator{\supp}{supp}
\DeclareMathOperator{\Bun}{Bun}

\DeclareMathOperator{\Qq}{\mathcal{Q}}
\DeclareMathOperator{\Qb}{\bar{\Qq}}

\DeclareMathOperator{\Pp}{\mathcal{P}}

\DeclareMathOperator{\Ll}{\mathcal{L}}
\DeclareMathOperator{\Tt}{\mathcal{T}}
\DeclareMathOperator{\Pb}{\bar{\mathcal{P}}}

\DeclareMathOperator{\id}{id}

\DeclareMathOperator{\OmegaX}{\Omega^1_X}

\DeclareMathOperator{\res}{res}

\DeclareMathOperator{\rk}{rk}
\DeclareMathOperator{\F}{\mathcal{F}}

\DeclareMathOperator{\Higgs}{\mathcal{M}_{Dol}}
\DeclareMathOperator{\adHiggs}{\mathcal{M}^{ad}_{Dol}}

\DeclareMathOperator{\ssLoc}{\mathcal{M}_{dR}^{s}}
\DeclareMathOperator{\ssHiggs}{\mathcal{M}_{Dol}^{s}}

\DeclareMathOperator{\G}{\mathbb{G}}

\DeclareMathOperator{\GL}{GL}

\DeclareMathOperator{\A}{\mathcal{A}}

\DeclareMathOperator{\M}{\mathcal{M}}
\DeclareMathOperator{\W}{\mathcal{W}}
\DeclareMathOperator{\X}{\mathcal{X}}
\DeclareMathOperator{\Y}{\mathcal{Y}}
\DeclareMathOperator{\Z}{\mathcal{Z}}
\DeclareMathOperator{\C}{\mathcal{C}}

\DeclareMathOperator{\Ee}{\mathcal{E}}

\DeclareMathOperator{\Gg}{\mathcal{G}}

\DeclareMathOperator{\Spec}{Spec}
\DeclareMathOperator{\Loc}{\mathcal{M}_{dR}}
\DeclareMathOperator{\HdR}{\chi_{dR}}
\DeclareMathOperator{\HDol}{\chi_{Dol}}

\DeclareMathOperator{\Sym}{Sym}

\DeclareMathOperator{\Oo}{\mathcal{O}}

\begin{document}

\newtheorem{defi}{Definition}[section]
\newtheorem{thm}[defi]{Theorem}
\newtheorem{prop}[defi]{Proposition}
\newtheorem{cor}[defi]{Corollary} 
\newtheorem{lemma}[defi]{Lemma}
\newtheorem{rmk}[defi]{Remark}
\newtheorem{cl}[defi]{Claim}
\newtheorem{ex}[defi]{Example}
\newtheorem{ass}[defi]{Assumption}

\title{Hilbert schemes as moduli of Higgs bundles and local systems}
\date{}
\author{Michael Groechenig}
\maketitle
\abstract{We construct five families of two-dimensional moduli spaces of parabolic Higgs bundles (respectively local systems) by taking the equivariant Hilbert scheme of a certain finite group acting on the cotangent bundle of an elliptic curve. We show that the Hilbert scheme of $m$ points of these surfaces is again a moduli space of parabolic Higgs bundles (respectively local systems), confirming a conjecture of Boalch in these cases and extending a result of Gorsky--Nekrasov--Rubtsov. Using the McKay correspondence we establish the autoduality conjecture for the derived categories of the moduli spaces of Higgs bundles under consideration.}
\newpage

\tableofcontents

\section{Introduction}

Let $X$ be a smooth projective curve defined over an algebraically closed field $k$. A Higgs bundle on $X$ is a pair consisting of a vector bundle $E$ on $X$ and a Higgs field $\theta$ given by a morphism of locally free sheaves $$\theta: E \to E \otimes \Omega_X^1,$$ where $\Omega_X^1$ denotes the sheaf of $1$-forms on $X$. The study of Higgs bundles on curves has been initiated by Hitchin (\cite{MR887284}) in relation with the construction of solutions to a dimensional reduction of the Yang-Mills equation. An essential device to study Higgs bundles has been the Hitchin map, which associates to a pair $(E,\theta)$ the characteristic polynomial of the Higgs field $\theta$. The zero scheme cut out by the characteristic polynomial in $T^*X$ is a possibly highly singular curve, referred to as the spectral curve. Passing to moduli spaces and fixing the rank $n$ of the vector bundle $E$ reveals a beautiful geometry. The moduli space of Higgs bundles $\M$ fibres over an affine space $\A$, giving rise to an integrable system (\cite{MR885778}) $$\chi: \M \to \A.$$ The fibres of the Hitchin map can be understood as compactified Picard varieties of the spectral curves. While this implies that the generic Hitchin fibre is an abelian variety, the global geometry of the Hitchin fibration remains mysterious.

Here we describe five infinite families of examples of moduli spaces of parabolic Higgs bundles (see Definition \ref{defi:parabolicHiggs}), whose geometry is non-trivial\footnote{in the sense that the Hitchin map is not everywhere smooth} but remains nonetheless manageable. Every family is constructed as an application of the following theorem, which will be proved as Theorem \ref{thm:Hilbert}.

\begin{thm}\label{thm:main}
Let $\M$ be the smooth moduli space of stable parabolic Higgs bundles of dimension two, with stability conditions, degree and dimension vector as specified in Theorem \ref{thm:toyexamples}, defined over an algebraically closed field $k$ of vanishing or sufficiently high characteristic. For every positive integer $m$ the Hilbert scheme of $m$ points $\M^{[m]}$ is again a moduli space of parabolic Higgs bundles, and the Hitchin morphism factors through the Hilbert-Chow map $\M^{[m]} \to \M^{(m)}$ to the symmetric product.
\end{thm}

Boalch conjectured in \cite[Rmk. 11.3]{Boalch:2011fk} that the Hilbert scheme of $m$ points of a two-dimensional moduli space of meromorphic Higgs bundles is again a moduli space of meromorphic Higgs bundles. It has been shown by Gorsky--Nekrasov--Rubtsov that Hilbert schemes of cotangent spaces of elliptic curves are moduli spaces of parabolic Higgs bundles (\cite[Sect. 5.1]{MR1859601}). The above Theorem \ref{thm:main} thus extends their result to other examples of two-dimensional moduli spaces of parabolic Higgs bundles.

The two-dimensional examples of moduli spaces of parabolic Higgs bundles, which we will consider, are in correspondence with the affine Dynkin diagrams $\widetilde{A}_0$, $\widetilde{D}_4$, $\widetilde{E}_6$, $\widetilde{E}_7$, and $\widetilde{E}_8$ and are constructed as certain resolutions of singularities of quotients of cotangent bundles of elliptic curves.

In the simplest case of $\widetilde{A}_0$ the corresponding moduli space in this list is the cotangent bundle of an elliptic curve $E$. For $\widetilde{D}_4$ one considers the natural action of the group $\Gamma = \mathbb{Z}/2\mathbb{Z}$ on $T^*E$ given by the inverse morphism $x \mapsto -x$ of $E$. The two-dimensional moduli spaces of this type are certain crepant resolutions of the GIT quotients $$T^*E/\Gamma,$$ given by the $\Gamma$-Hilbert schemes of $T^*E$.

In the $\widetilde{E}_7$-case we pick an elliptic curve $E$ with an automorphism of order $4$. This is an elliptic curve with a special form of complex multiplication, which corresponds to the lattice of Gaussian integers $\mathbb{Z}[i] \subset \mathbb{C}$. According to \cite[p. 483]{MR1312368} this elliptic curve corresponds to the equation 
$$y^2 = x^3 + x,$$
and the $\mathbb{Z}/4\mathbb{Z}$-action is generated by $(x,y) \mapsto (-x,iy)$. Again we construct a two-dimensional moduli space of parabolic Higgs bundles by taking the $\Gamma$-Hilbert scheme of $T^*E$.

For the graphs $\widetilde{E}_6$ and $\widetilde{E}_8$ we proceed analogously with an elliptic curve $E$ with a non-trivial $\mathbb{Z}/3\mathbb{Z}$- and $\mathbb{Z}/6\mathbb{Z}$-action respectively. Over the field of complex numbers such a curve is given by the lattice of Eisenstein integers $\mathbb{Z}[\omega] \subset \mathbb{C}$, where $\omega$ is a primitive third root of unity. According to \cite[p. 483]{MR1312368} an explicit equation for this curve is given by $$y^2 + y = x^3.$$ The $\mathbb{Z}/6\mathbb{Z}$-action is generated by $\mathbb{Z}/3\mathbb{Z}$-action given by $\xi\cdot(x,y) = (\xi x,y)$ for every third root of unity $\xi$, and the $\mathbb{Z}/2\mathbb{Z}$-action induced by the inverse map of $E$.

The following theorem will be proved as Theorem \ref{thm:toyexamples}.

\begin{thm}\label{thm:main2}
Let $E$ be an elliptic curve with an action of a finite group $\Gamma$ as described above. Then there is  an isomorphism of the $\Gamma$-Hilbert scheme of $T^*E$ with a two-dimensional moduli space of stable parabolic Higgs bundles. In particular these moduli spaces are crepant resolutions of the GIT quotients $T^*E/\Gamma$.
\end{thm}

The proofs of the Theorems \ref{thm:main} and $\ref{thm:main2}$ rely heavily on the machinery of the Fourier-Mukai transform. The use of these techniques in the theory of moduli is by no means original; the importance of the Fourier-Mukai transform for this field of geometry has already been realized in Mukai's paper \cite{MR946249}. While the arguments of Gorsky--Nekrasov--Rubtsov in \cite[Sect. 5.1]{MR1859601} were mainly based on a gauge theoretical foundation, a closer look at their work reveals that it uses the Fourier-Mukai transform.

As has been mentioned above, the generic fibre of the Hitchin fibration $\chi: \M \to \A$ is a self-dual abelian variety. In particular there exists a Zariski open subset $\A^{\sm} \subset \A$ corresponding to smooth spectral curves and a line bundle $\Pp$ on $\M^{\sm} \times_{\A^{\sm}} \M^{\sm}$ inducing a non-trivial autoequivalence $$D^b_{coh}(\M^{\sm}) \to D^b_{coh}(\M^{\sm}).$$ It is an important open problem to determine whether this equivalence extends to the whole base $\A$, where $\M$ is either the moduli space of semistable Higgs bundles or the full moduli stack. We refer to Conjecture 2.5 in \cite{Donagi:fk} for a more precise statement of this conjecture. Recently this has been achieved by Arinkin in \cite{Arinkin:2010uq} over the locus of integral spectral curves. 

To demonstrate the flexibility of our Theorems \ref{thm:main} and \ref{thm:main2} we prove this so-called autoduality conjecture for the moduli space of Higgs bundles in all the cases given by Theorem \ref{thm:main} (see Theorem \ref{thm:autoduality} for a proof).

\begin{thm}\label{thm:main3}
Let $\M$ be one of the two-dimensional moduli spaces of parabolic Higgs bundles considered in Theorem \ref{thm:main}. We denote its Hitchin base by $\A_1$. According to Theorem \ref{thm:main} the Hilbert scheme $\M^{[m]}$ is again a moduli space of parabolic Higgs bundles, with Hitchin base denoted by $\A_m$. There exists a derived equivalence of Fourier-Mukai type $$D^b_{coh}(\M^{[m]}) \cong D^b_{coh}(\M^{[m]}),$$ defined relative to the Hitchin base $\A_m$ and extending the Fourier-Mukai transform over the locus of smooth spectral curves.
\end{thm}

The main ingredient of the proof of Theorem \ref{thm:main3} is the derived McKay correspondence as shown in \cite{Bridgeland:fk}. Interestingly both Arinkin's proof of the autoduality over the integral locus in \cite{Arinkin:2010uq} and our proof of Theorem \ref{thm:main3} rely on Haiman's work on the $n!$-conjecture (\cite{MR1839919}).

All of our results described above have a counterpart for moduli spaces of flat connections, where the role of the Fourier-Mukai transform of the Hitchin system is now taken by the Geometric Langlands correspondence. For the sake of avoiding repetition we refer the reader to the Theorems \ref{thm:flattoyexamples}, \ref{thm:localHilbert} and \ref{thm:localduality} instead.\\

\noindent\textit{Acknowledgements.} I would like to thank my supervisor Tam\'as Hausel for explaining the structure of two-dimensional moduli spaces of parabolic Higgs bundles to me, for introducing me to Boalch's conjecture, and for his remarks on this paper. Moreover I would like to thank Philip Boalch and Carlos Simpson for their interest in this work and for useful comments. I thank Tim Adamo and Tom Sutherland for their comments on a draft of this paper. This work has been funded by an EPSRC grant received under the contract EP/G027110/1.

\section{Higgs bundles and local systems}

In this section we will review the theory of Higgs bundles, local systems and parabolic structures. From now on $k$ denotes an algebraically closed field of either vanishing or sufficiently high characteristic. We assume that all schemes and related constructions are defined over $k$.

\subsection{Higgs bundles and the BNR correspondence}

In the following we fix a smooth projective curve $X$ and a positive integer $n \in \mathbb{N}$. For a scheme $S$ we define the notion of an $S$-family of Higgs bundles below. Of particular importance is the case of a $k$-family, which simply is a Higgs bundle on $X$. We denote by $\OmegaX$ the cotangent sheaf on $X$ and by $p_X: X \times S \to X$ the canonical projection.

\begin{defi}\label{def:Higgs}
An $S$-family of Higgs bundles on $X$ consists of a locally free sheaf $\Ee$ of rank $n$ on $X \times S$ and a \emph{Higgs field} given by a morphism of $\Oo_X$-modules $$ \theta: \Ee \to \Ee \otimes p_X^*\OmegaX. $$
This gives rise to a moduli stack of Higgs bundles denoted by $\Higgs(X)$, which sends $S$ to the groupoid of $S$-families of Higgs bundles on $X$.
\end{defi}

Higgs bundles on $X$ can be perceived as certain coherent sheaves on the cotangent bundle $T^*X$. This is known as the \emph{Beauville--Narasimhan--Ramanan}-correspondence (\cite{MR998478}) and was proved in full generality in \cite[Lemma 6.8]{MR1320603}. We will only need a weak version of the BNR correspondence, which is stated and proved below. The map $\pi: T^*X \to X$ is the canonical projection, and $\pi_S$ the base change $$T^*({X\times S/S}) \to X \times S.$$

\begin{prop}[weak BNR correspondence]\label{prop:BNR}
There is a natural equivalence between the groupoid of $S$-families of Higgs bundles $(\Ee,\theta)$ and the groupoid of quasi-coherent $S$-flat sheaves $\mathcal{F}$ on the relative cotangent bundle $T^*(X \times S/S)$ satisfying $$ \pi_{S,*}\mathcal{F} = \Ee. $$
\end{prop}

\begin{proof}
Let us denote by $\Theta_X$ the tangent sheaf of $X$. The Higgs field $\theta$ of a Higgs pair $(\Ee,\theta)$ gives rise to a morphism of sheaves $$\Ee \otimes p_X^*\Theta_X \to \Ee,$$ which in turn induces a morphism $$\Sym^{\bullet} p_X^*\Theta_X \otimes \Ee \to \Ee.$$
This endows $\Ee$ with the structure of a module over the algebra $$\Sym^{\bullet} p_X^*\Theta_{X} = \pi_{S,*}\Oo_{T^*_{X \times S/S}}.$$ Since $\pi_S$ is an affine morphism this gives rise to a $S$-flat quasi-coherent sheaf $\mathcal{F}$ on $T^*_{X \times S /S}$, such that $\pi_{S,*}\mathcal{F} = \Ee$.

Vice versa, given an $S$-flat quasi-coherent sheaf $\mathcal{F}$ on $T^*_{X\times S/S}$ pushing-forward to a locally free sheaf $\pi_{S,*}\mathcal{F} =: \Ee$, we can define a Higgs field $\theta$ as follows: the push-forward $\Ee$ is endowed with the structure of an $\Sym^{\bullet} p_X^*\Theta_{X}$-module, in particular we have a map 
$$ p_X^*\Theta_X \otimes \Ee \to \Ee, $$
giving rise to a Higgs field $\theta: \Ee \to \Ee \otimes p_X^*\OmegaX$. 
\end{proof}

If we perceive the Higgs field $\theta$ as a twisted endomorphism of $\Ee$ we see that the expression $$a(\lambda) := \det(\lambda - \theta)$$ is well-defined and is a polynomial $$\lambda^n + a_{n-1} \lambda^{n-1} + \dots + a_0,$$ where $a_i \in H^0(X,\Omega_X^{\otimes (n-i)})$. 

\begin{defi}
Let $\A$ be the affine space associated to the vector space $$\bigoplus_{i=0}^{n-1}H^0(X,\Omega_X^{\otimes (n-i)}),$$ it is called the \emph{Hitchin base}. The morphism of stacks $$\HDol: \Higgs(X) \to \A,$$ sending a Higgs bundle $(\Ee,\theta)$ to the characteristic polynomial $a(\lambda)$ of $\theta$ is called the Hitchin morphism.
\end{defi}

\subsection{Parabolic vector bundles}\label{ssec:parabolic}

As before $X$ denotes a smooth projective curve of genus $g$. We assume that $D = n_1p_1+ \dots + n_k p_k$ is an effective divisor on $X$, i.e. $n_i > 0$ for all $i$. The tuple $\widehat{X} = (X,D)$ will be referred to as a \emph{weighted curve}. We fix a positive integer $n \in \mathbb{N}$ and for $i=1,\dots,k$ partitions $$n = \lambda_{i0} \geq \cdots \geq \lambda_{in_i} = 0.$$ We will also write $l(\lambda_i) = n_i$ to denote the length of a partition. Following \cite[Sect. 2.2]{Hausel:2008ys} this numerical data will be encoded in the following diagram:

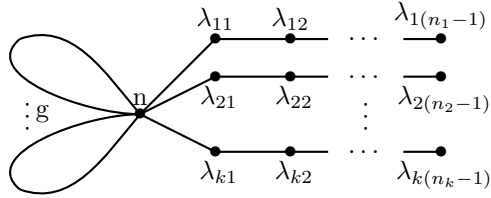
\begin{center}
\begin{tikzpicture}
\node at (2,0) {$\bullet$};
\node [above] at (2,0) {n};
\node at (3,1) {$\bullet$};
\node[above] at (3,1) {$\lambda_{11}$};
\node at (4,1) {$\bullet$};
\node[above] at (4,1) {$\lambda_{12}$};
\node at (5,1) {$\dots$};
\node at (6,1) {$\bullet$};
\node[above] at (6,1) {$\lambda_{1(n_1 - 1)}$};
\draw [thick] (2,0) -- (3,1);
\draw [thick] (3,1) -- (4,1);
\draw [thick] (4,1) -- (4.5,1);
\draw [thick] (5.5,1) -- (6,1);

\node at (3,0.5) {$\bullet$};
\node[below] at (3,0.5) {$\;\lambda_{21}$};
\node at (4,0.5) {$\bullet$};
\node[below] at (4,0.5) {$\;\lambda_{22}$};
\node at (5,0.5) {$\dots$};
\node at (6,0.5) {$\bullet$};
\node[below] at (6,0.5) {$\;\lambda_{2(n_2-1)}$};
\draw [thick] (2,0) -- (3,0.5);
\draw [thick] (3,0.5) -- (4,0.5);
\draw [thick] (4,0.5) -- (4.5,0.5);
\draw [thick] (5.5,0.5) -- (6,0.5);

\node at (5,0) {\rotatebox{90}{$\dots$}};

\node at (3,-0.5) {$\bullet$};
\node[below] at (3,-0.5) {$\;\lambda_{k1}$};
\node at (4,-0.5) {$\bullet$};
\node[below] at (4,-0.5) {$\;\lambda_{k2}$};
\node at (5,-0.5) {$\dots$};
\node at (6,-0.5) {$\bullet$};
\node[below] at (6,-0.5) {$\;\lambda_{k(n_k-1)}$};
\draw [thick] (2,0) -- (3,-0.5);
\draw [thick] (3,-0.5) -- (4,-0.5);
\draw [thick] (4,-0.5) -- (4.5,-0.5);
\draw [thick] (5.5,-0.5) -- (6,-0.5);

\draw[thick] (2,0) to [out=130,in=20] (0.4,1);
\draw[thick] (0.4,1) to [out=220,in=180] (2,0);

\node at (0.5,0) {\rotatebox{90}{$\dots$}};
\node at (0.7,0) {g};

\draw[thick] (2,0) to [out=230,in=340] (0.4,-1);
\draw[thick] (0.4,-1) to [out=140,in=180] (2,0);
\end{tikzpicture}
\end{center}

Let $S$ be a scheme. Below we define the notion of an $S$-family of parabolic vector bundles. We denote by $\iota_i: S \to X \times S$ the base change of the closed immersion $$p_i: \Spec k \to X,$$ corresponding to the marked points $p_i \in X(k)$.

\begin{defi}
An $S$-family of parabolic vector bundles on $\widehat{X}$ consists of a locally free sheaf $\Ee$ of rank $n$ on $X \times S$ together with flags $$ 0 = F_{in_i} \subset F_{i(n_i-1)} \subset \cdots \subset F_{i0} = \iota_i^*\Ee $$ of locally free subsheaves of $\iota_i^*\Ee$ of rank $$\rk F_{ij} = \lambda_{ij},$$ such that the successive quotients are locally free. The moduli stack of parabolic vector bundles will be denoted by $\Bun(\widehat{X}) = \Bun(\widehat{X},n,\lambda_{\bullet})$
\end{defi}

Often the definition of parabolic bundles includes weights (see Definition \ref{defi:stable}) as part of the data; we prefer to view them as part of the stability condition.

Given a vector bundle $E$ on $X$, a point $x \in X(k)$, and a subspace $L \subset E_x$ of the fibre $E_x = E/E(-x)$, we can define a locally free sheaf $E_L$ by the formula
$$E_L := \ker (E \to E_x/L),$$ where $E \to E_x/L$ is the obvious map factoring through $E \to E_x$. This process can be reversed, since $$L = \ker(E_x \to \coker (E_L \to E)).$$

The process described above gives rise to an alternative description of parabolic vector bundles as flags of locally free sheaves: For $i=1,\dots,k$ and $j = -n+1,\dots,0$ we define $E_{ij}$ to be the locally free subsheaf of $E$ given by $F_{j+n} \subset E_{p_i}$. For arbitrary $j \in \mathbb{Z}$ we can write $j = j' + m$, where $-n < j' \leq 0$ and $m \in \mathbb{Z}$, and define $E_{ij} := E_{ij'}\otimes \Oo_X(mp_i)$. We conclude that there is an alternative description of parabolic bundles in terms of nested sequences of locally free sheaves (\cite[Sect. 3]{MR1040197}). 

\begin{lemma}\label{lemma:sequences}
The stack $\Bun(\widehat{X},n,\lambda_{\bullet})$ is equivalent to the stack of families of $\mathbb{Z}^k$-indexed sequences of locally free sheaves $(V_{i})_{i \in \mathbb{Z}^k}$ on $X$ satisfying $$V_i \subset V_{i+e_j},$$ where $(e_j)$ denotes the canonical basis of $\mathbb{Z}^k$, and $$V_{i+n_je_j} = V_i \otimes \Oo_X(p_j).$$
\end{lemma}

The interpretation of parabolic vector bundles as sequences of locally free sheaves $\mathcal{V} = (V_{i})_{i \in \mathbb{Z}^k}$ suggests a definition for the \emph{dual} parabolic vector bundle $\mathcal{V}^{\dual}$ given by the sequence $$V^{\dual}_i := (V_{-i})^{\dual}.$$ This description of the dual parabolic bundle is easily seen to be compatible with the following definition.

\begin{defi}\label{defi:dual}
Let $(\widehat{E},F_{\bullet\bullet})$ be a parabolic bundle on $\widehat{X}$. The dual parabolic bundle $\widehat{E}^{\dual}$ has underlying vector bundle $E^{\dual}$ and flag data given by $$F^{\dual}_{ij} = \ker(E^{\dual}_{x_i} \to F^{\dual}_{i(n-j)}).$$
\end{defi}

\begin{lemma}\label{lemma:dual}
If $\widehat{E} = (\widehat{E},F_{\bullet\bullet})$ is a parabolic bundle on $\widehat{X}$ corresponding to the sequence $\mathcal{V} = (V_{i})_{i \in \mathbb{Z}^k}$ then the dual $\widehat{E}^{\dual}$ corresponds to the dual sequence $\mathcal{V}^{\dual} = ((V_{-i})^{\dual})_{i \in \mathbb{Z}^k}$.
\end{lemma}

\begin{proof}
For every $i = 1,\dots, k$ we denote by $e_i$ the canonical basis element in $\mathbb{Z}^k$. We need to compute $$\ker(E^{\dual}_{x_i} \to \coker(V^{\dual}_{-je_i} \to E^{\dual}))$$ for $j = 0, \dots, n_i-1$. But $\coker(V^{\dual}_{-je_i} \to E^{\dual})) = \coker(E \to V_{je_i})^{\dual} = F^{\dual}_{i(n-j)}$.
\end{proof}

Stability conditions for parabolic vector bundles are parametrized by finite increasing sequences of positive real numbers $(\alpha_{ij}) \in [0,1)$, where $i = 1, \dots, k$ and $j = 0, \dots n_i-1$. Following \cite{MR1411302} we denote by $m_{ij} = \lambda_{ij} - \lambda_{ij+1}$ and define the parabolic degree of a parabolic vector bundle $\widehat{E}$ to be $$\deg \widehat{E} = \deg E + \sum_{i=1}^k \sum_{j=0}^{n_i-1}\alpha_{ij}m_{ij}.$$ The parabolic slope of $\widehat{E}$ is given by $$\mu(\widehat{E}) = \frac{\deg \widehat{E}}{\rk E}.$$

\begin{defi}\label{defi:stable}
A parabolic vector bundle $\widehat{E}$ is said to be stable if for every proper subbundle $F$ of $E$ with the induced parabolic structure, we have $$\mu(\widehat{F}) < \mu(\widehat{E}).$$ If we only have $\mu(\widehat{F}) \leq \mu(\widehat{E})$ for every proper subbundle $F$ we say that $\widehat{E}$ is semistable.
\end{defi}

\subsection{Parabolic Higgs bundles}\label{ssec:parHiggs}

\begin{defi}\label{defi:parabolicHiggs}
An $S$-family of parabolic Higgs bundles on $\widehat{X}$ consists of an $S$-family of parabolic vector bundles $(\Ee,\mathcal{F}_{\bullet\bullet})$ and a parabolic Higgs field given by a morphism of $\Oo_X$-modules $$\theta: \Ee \to \Ee \otimes p_X^*\OmegaX(p_1 + \dots p_k). $$
The latter is required to satisfy the condition that $$\res_{p_i}\theta(\mathcal{F}_{ij})\subset \mathcal{F}_{ij+1}.$$ A parabolic Higgs bundle is stable (respectively semistable) if the condition of Definition \ref{defi:stable} is satisfied for all proper subbundles $F$ preserved by the Higgs field $\theta$. The moduli stack of parabolic Higgs bundles will be denoted by $\Higgs(\widehat{X}) = \Higgs(\widehat{X},n,\lambda_{\bullet})$. The moduli space of stable parabolic Higgs bundles $\ssHiggs$ is constructed as the rigidification of the open substack of stable Higgs bundles of $\Higgs(X)$.
\end{defi}

As before there exists a Hitchin base $$\A := \bigoplus_{i=0}^{n-1}H^0(X,\Omega_X^{i+1}(ip_1 + \cdots + ip_k))$$ and a Hitchin morphism $$\HDol: \Higgs(\widehat{X}) \to \A,$$ which sends a parabolic Higgs bundle to the characteristic polynomial of its Higgs field $\theta$. Using a method developed by Langton, Yokogawa shows in Corollary 5.13 and 1.6 of \cite{MR1231753} that the Hitchin map induces a proper map on the moduli space of parabolic Higgs bundles. In the case of Higgs bundles without parabolic structures this is a Theorem of Nitsure (\cite[Thm. 6.1]{MR1085642}).

\begin{thm}[Nitsure, Yokogawa]
Let $\Higgs(\widehat{X})$ be a moduli space of semistable parabolic Higgs bundles of a fixed type. Then the Hitchin map $$\chi: \Higgs(\widehat{X}) \to \A$$ is proper.
\end{thm}

\subsection{Parabolic local systems}

Closely related to the theory of Higgs bundles is the theory of local systems (i.e. vector bundles with a flat connection). The work of  Hitchin, Simpson and others (\cite{MR887284}, \cite{MR1040197}, \cite{MR1179076}) has exhibited a natural hyperk\"ahler structure on the moduli space of stable Higgs bundles defined over a complex curve. One obtains the moduli space of irreducible local systems by hyperk\"ahler rotation. In particular there is a canonical diffeomorphism relating the moduli space of stable Higgs bundles and irreducible local systems.

Using the notation from subsection \ref{ssec:parHiggs} we define the notion of an $S$-family of parabolic local systems.

\begin{defi}\label{def:parLoc}
An $S$-family of parabolic local systems on $\widehat{X}$ consists of an $S$-family of parabolic vector bundles $(\Ee,\mathcal{F}_{\bullet\bullet})$ and a parabolic flat connection given by a morphism of $k$-linear sheaves $$\nabla: \Ee \to \Ee \otimes p_X^*\OmegaX(p_1 + \dots p_k), $$
satisfying the Leibniz identity and the condition that $$\res_{p_i}\nabla(\mathcal{F}_{ij})\subset \mathcal{F}_{ij},$$ with eigenvalues of the residue given by the canonical weights from Lemma \ref{lemma:canonicalweights}. A parabolic local system is stable (respectively semistable) if the condition of Definition \ref{defi:stable} is satisfied for all proper subbundles $F$ preserved by the connection $\nabla$. 
The moduli stack of parabolic local systems will be denoted by $\Loc(\widehat{X}) = \Loc(\widehat{X},n,\lambda_{\bullet})$. If $\widehat{X} = X$, i.e. there are no marked points, we will simply speak of local systems on the curve $X$. The moduli space of stable parabolic local systems $\ssLoc$ is constructed as the rigidification of the open substack of stable local systems of $\Loc(X)$.
\end{defi}

\subsection{Orbifolds and parabolic structures}

Given a weighted curve $\widehat{X}$ we can associate to it an orbifold $\widetilde{X}$, assuming that the characteristic of $k$ is zero or large enough. We emphasize that the word orbifold refers to a smooth Deligne-Mumford stack (\cite[Def. 4.1]{MR1771927}) in our context.

The orbifold $\widetilde{X}$ is defined by the following glueing data: let $\mathbb{D}_x$ denote the formal disc $\Spec \widehat{\Oo}_x$ around a point $x \in X(k)$. Given an effective divisor $D \subset X$ represented by the effective linear combination $n_1p_1 + \dots n_k p_k$, we let $\mathbb{D}_i := \mathbb{D}_{p_i}$ and $$U := X - D.$$ The fibre product $U \times_X \mathbb{D}_i$ is given by the punctured formal disc $$ \mathbb{D}^{\bullet}_i := \mathrm{Frac}\; \widehat{\Oo}_x. $$
Let us denote by $$[n]: \Spec k((t)) \to \Spec k((t))$$ the faithfully flat morphism given by $t \mapsto t^n$. Note that this is an \'etale morphism if and only if the characteristic $p$ of $k$ does not divide $n$. Picking a formal coordinate $t_i$ around every point $p_i$ we obtain morphisms \begin{center}\begin{tikzpicture}\node at (0,0) {$[n_i]:$};\node at (1,0) {$\mathbb{D}^{\bullet}_i$};\node at (2,0) {$\mathbb{D}^{\bullet}_i$}; \draw[->,thick] (1.3, 0) -- (1.7,0); \end{tikzpicture}\end{center} for every $i = 1, \dots, n$.

On the disc $\mathbb{D}_i$ we define the obvious action of the group $\mu_{n_i}$ of $n_i$-th roots of unity by multiplication. Using this action we glue the quotient stacks $[\mathbb{D}_i/\mu_{n_i}]$ back to $U$ using the $\mu_{n_i}$-equivariant maps induced by the $[n_i]$
\begin{center}
\begin{tikzpicture}

\node at (0,0) {$\mathbb{D}^{\bullet}_i$};
\node at (1,0) {$\mathbb{D}^{\bullet}_i$};
\node at (2,0) {$U.$};
\node at (0.5,0.3) {$[n_i]$};

\draw [->,thick] (0.3,0) -- (0.7,0);
\draw [->,thick] (1.3,0) -- (1.7,0);

\end{tikzpicture}
\end{center}
According to Theorem 6.1 in \cite{MR0399094} this defines an algebraic stack independently of the characteristic $p$ of $k$. Nonetheless this is a smooth Deligne-Mumford stack if either $p = 0$ or $\gcd(p,n_i) = 1$ for all $i$.

\begin{ass}\label{ass}
The field $k$ is algebraically closed and its characteristic $p$ is either zero or satisfies $\gcd(p,n_i) = 1$ for all $i$. 
\end{ass}

It is a result of Furuta--Steer (\cite[sect. 5]{MR1185787}) that vector bundles on the orbifold $\widetilde{X}$ translate into parabolic vector bundles on the weighted curve $\widehat{X}$. Nasatyr--Steer (\cite[Sect. 5A]{MR1375314}) discuss the analogous result for Higgs bundles. The local systems case is treated in \cite{Biswas:2012uq} by Biswas--Heu. In the remaining part of this subsection we explain how this is proved in the realm of algebraic geometry instead of the analytic theory of Riemann surfaces used in \cite{MR1185787} and \cite{MR1375314}.

Orbicurves as considered here can also be seen as certain root stacks associated to weighted curves (and this is what we will be doing implicitly). The correspondence described above is therefore reminiscent of a correspondence between parabolic vector bundles and vector bundles on root stacks, as established by N. Borne in \cite{borne}. We refer to \emph{loc. cit.} for algebraic proofs of the statements below.

The correspondence between vector bundles on the orbifold $\widetilde{X}$ and parabolic bundles on $\widehat{X}$ is based on the following two observations. The natural morphism
$$\tau: \widetilde{X} \to X$$
realizes $X$ as the coarse moduli stack for the Deligne-Mumford stack $\widetilde{X}$. The second observation is that the functor $\tau_*$ from quasi-coherent sheaves on $\widetilde{X}$ to quasi-coherent sheaves on $X$ is not faithful. Nonetheless it sends a vector bundle $\widetilde{E}$ on $\widetilde{X}$ to a vector bundle $E := \tau_*\widetilde{E}$, since every torsion-free sheaf on a smooth curve is locally free.

\begin{ex}\label{ex:disc}
Let $\mu_r$ be the cyclic group $r$-th roots of unity. It acts on $\mathbb{D}:= \Spec k[[t]]$ via $\xi \cdot t = \xi t$, where $\xi$ is an $r$-th root of unity. If $\widetilde{X}$ is the quotient stack $$[\mathbb{D}/\mu_r]$$ we can identify the coarse moduli space $X$ with $\Spec k[[t^r]]$. The functor $\tau_*$ sends a $\mu_r$-equivariant $k[[t]]$-module $M$ to the $k[[t^r]]$-module $M^{\mu_r}$.
\end{ex}

To reconcile the loss of information under the map $\widetilde{E} \mapsto E$ we define a $\mathbb{Z}$-indexed sequence of line bundles $(L_i)_{i\in \mathbb{Z}}$ for every orbifold point of the orbifold $\widetilde{X}$, satisfying $$L_i \subset L_{i+1}$$ for all $i \in \mathbb{Z}$, and send $\widetilde{E}$ to the parabolic vector bundle associated to the filtered locally free sheaf $(\tau_*(\widetilde{E} \otimes L_i))_{i \in \mathbb{Z}}$.

\begin{defi}
Let $\widehat{X}$ be a weighted curve and $\widetilde{X}$ the associated orbicurve. For every marked point $p_i$ of $\widehat{X}$ we pick an $n$-th root $L_{i1}$ of $\tau^*\Oo_X(p_j)$. The line bundle $L_{ij}$ is defined to be $$L_{ij} := L_{i1}^j.$$
\end{defi}

The existence of $L_{i1}$ can be seen locally on $X$ using the notation of Example \ref{ex:disc}. Let $x$ be the origin of the disc $\mathbb{D}$. Since $\tau^*\Oo_X(x)$ is given by the $k[[t]]$-module $t^{-n}k[[t]]$, we see that $t^{-1}k[[t]]$ is an $n$-th root of $\tau^*\Oo_X(x)$.

By a formal-disc argument we can show the following remark:

\begin{rmk}
Let $n_i$ denote the order of the stabilizer group of the point $x_i$, respectively the weight of $p_i$. Then we have  $\tau^*\Oo_X(p_i) = L_{in_i}$.
\end{rmk}

Using this remark and Lemma \ref{lemma:sequences} it is a consequence of the projection formula $$\tau_*\tau^*\Oo_X(p_i) \cong \Oo_X(p_i) \otimes \tau_*\Oo_{\widetilde{X}} \cong \Oo_X(p_i)$$ that the sequence of locally free sheaves $$E_{ij} := \tau_*(\widetilde{E} \otimes L_{ij})$$ gives rise to a parabolic vector bundle $\widehat{E} := (E,F_{\bullet \bullet})$ on $\widehat{X}$. We denote the map sending an orbibundle $\widetilde{E}$ to the parabolic bundle $\widehat{E}$ by $A$. 

\begin{prop}[Furuta--Steer]\label{prop:parabolic-orbibundles}
The association $$A: \widetilde{E} \mapsto \widehat{E}$$ described above gives rise to an equivalence of groupoids of vector bundles on the orbicurve $\widetilde{X}$ and parabolic vector bundles on the weighted curve $\widehat{X}$.
\end{prop}

As a parabolic bundle is obtained from an orbibundle by push-forward, one should expect the respective degrees of the bundles to be related. A Riemann-Roch computation (using \cite[Cor. 4.14]{MR1710187} for Deligne-Mumford stacks) reveals the precise relation between the two degrees.

\begin{lemma}\label{lemma:canonicalweights}
Under the equivalence of Proposition \ref{prop:parabolic-orbibundles} the degree of an orbibundle $\widetilde{E}$ is equal to the parabolic degree of the parabolic bundle $\widehat{E}$ with respect to the so-called \emph{canonical weights} $\alpha_{ij} := \frac{j}{n_i}$.
\end{lemma}

In the following remark we make the above correspondence more explicit using the notation of Example \ref{ex:disc}.

\begin{rmk}\label{rmk:explicit}
Let $E$ be a $\mu_r$-equivariant vector bundle on $\mathbb{D}$. The projection formula implies that we have $$E^{\Gamma} \otimes \Oo/\Oo(-x) \cong (E \otimes \Oo/L^{-r})^{\Gamma}.$$ Using this we may identify the corresponding parabolic vector bundle $\widehat{E}$ on $\widehat{\mathbb{D}}$ with the one given by the vector bundle $E^{\Gamma}$ together with the flags $$F_i := (E \otimes L^{-i}/L^{-r})^{\Gamma} \subset (E \otimes \Oo/L^{-r})^{\Gamma} \cong E^{\Gamma} \otimes \Oo/\Oo(-x).$$
\end{rmk}

The next lemma and its proof should clarify how the equivariant structure of an vector bundle on an orbicurve is translated into the flag data of a parabolic vector bundle.

\begin{lemma}\label{lemma:isotropy-representation}
Let $\Gamma = \mu_r$ be the finite cyclic group of order $r$ acting on $\mathbb{D} = \Spec k[[t]]$ through $\xi \cdot t = \xi t$, where $\xi$ is an $r$-th root of unity. Then the isomorphism classes of rank $n$ parabolic vector bundles on the weighted curve $\widehat{[\mathbb{D}/\Gamma]}$ correspond to isomorphism classes of representations of $\Gamma$ on an $n$-dimensional vector space. The regular representation of $\Gamma$ corresponds to a rank $r$ vector bundle with parabolic structure given by a complete flag. 
\end{lemma}

\begin{proof}
Let us denote by $\chi$ the character associated to the zero fibre of the line orbibundle $L$. By assumption we have $\chi(\xi) = \xi^{-1}$. If $E$ is a bundle on $[\mathbb{D}/\Gamma]$ and $(E_i)_{i\in \mathbb{Z}}$ denotes the corresponding parabolic bundle on $\widehat{[\mathbb{D}/\Gamma]}$. A section $s$ of $E_i$ non-vanishing at $0 \in \mathbb{D}$ corresponds to a $\Gamma$-invariant section of $E \otimes L^i$. This gives rise to an eigenline in $E_0$ on which $\Gamma$ acts by $\chi^{-i}$.

Vice versa given an eigenline $\langle v\rangle \subset (E)_0$ on which $\Gamma$ acts by $\chi^k$ then this gives rise to an eigenline in $(E \otimes L^{-k})_0$, on which $\Gamma$ acts trivially. This in turn gives rise to a section of $E_k$. We see that the parabolic structure encodes the $\Gamma$-action on the zero fibre $E_0$.

To verify the last assertion we only have to observe that the regular representation of $\Gamma$ is the direct sum $$\bigoplus_{k=0}^r V_{\chi^k},$$ where $V$ is a one-dimensional vector space with $\Gamma$ acting on it through the character specified in the subscript.
\end{proof}

As a next step we investigate what happens to extra structures like a Higgs field or a connection under the transition $\widetilde{E} \mapsto \widehat{E}$. The Definitions \ref{def:Higgs} and \ref{def:parLoc} are \'etale local in nature with respect to the curve $X$, in particular this allows us to make sense of Higgs bundles and local systems on an orbicurve $\widetilde{X}$.

\begin{prop}[Nasatyr--Steer, Biswas--Heu]\label{prop:orbi-parabolic}
Under the correspondence of Proposition \ref{prop:parabolic-orbibundles} a Higgs field $\widetilde{\theta}$ on an orbibundle $\widetilde{E}$ gets transformed to a parabolic Higgs field $\widehat{\theta}$ on $\widehat{E}$. Similarly a flat connection $\widetilde{\nabla}$ is sent to a parabolic flat connection $\widehat{\nabla}$. This defines a natural equivalence of groupoids between $S$-families of Higgs bundles (resp. local systems) on the orbicurve $\widetilde{X}$ and $S$-families of parabolic Higgs bundles (resp. parabolic local systems) on the weighted curve $\widehat{X}$.
\end{prop}

\section{Derived equivalences}

This section is a collection of technical results on derived categories that will be of use later. The geometrically-minded reader is encouraged to skip it and come back to it as required.

\subsection{Derived categories}

We begin by reviewing the theory of quasi-coherent sheaves and their derived categories on stacks. A good summary of this theory, together with theoretical justification for some of the definitions given below, is contained in \cite[Sect. 2]{Arinkin:2009yq}. 

The data of a quasi-coherent sheaf $F$ on a prestack $\X$ is equivalent to a collection of quasi-coherent sheaves $F_{U \to \X}$ for every affine scheme $U$ with a morphism $U \to \X$, in a way compatible with pullback. This compatibility condition stipulates the existence of isomorphisms $$\phi_{V \to U}: \psi^*F_{U \to \X} \to F_{V \to \X}$$ for every morphism $\psi: V \to U$ of $\X$-schemes, which are required to obey a compatibility law of their own. In the language of category theory we have exhibited the category of quasi-coherent sheaves on $\X$ as the $2$-limit of the categories $\mathrm{QCoh}(U)$ of quasi-coherent sheaves on $U$ $$\mathrm{QCoh}(\X) := \lim_{U \in \mathrm{Aff}/\X} \mathrm{QCoh}(U).$$
If $\X$ is an algebraic stack it is possible to replace the above $2$-limit by a less intimidating one. Let $Y \to \X$ be an atlas, i.e. a smooth surjective morphism, where $Y$ is a scheme. Faithfully flat descent theory implies that $\mathrm{QCoh}(\X)$ is equivalent to the $2$-limit

\begin{center}
\begin{tikzpicture}
\node at (0,0) {$\mathrm{QCoh}(\X)$};
\node at (1,0) {$\cong$};
\node at (1.52,0.03) {$\lim$};

\node at (2.7,0) {$\mathrm{QCoh}(Y)$};
\node at (5.5,0) {$\mathrm{QCoh}(Y \times_{\X} Y)$};
\node at (9.3,0) {$\mathrm{QCoh}(Y \times_{\X} Y \times_{\X} Y)$};

\node at (1.87,0) {[};
\node at (11.075,0) {].};

\draw [->,thick] (3.6,-0.1) -- (4.1,-0.1);
\draw [->,thick] (3.6,0.1) -- (4.1,0.1);

\draw [->,thick] (6.9,0) -- (7.4,0);
\draw [->,thick] (6.9,0.16) -- (7.4,0.16);
\draw [->,thick] (6.9,-0.16) -- (7.4,-0.16);
\end{tikzpicture}
\end{center}

This $2$-limit amounts to the simple fact that the data of a quasi-coherent sheaf on $\X$ is equivalent to a quasi-coherent sheaf $F_Y$ on the atlas $Y$ endowed with descent data. In the special case that $\X$ is a global quotient stack $[Y/G]$, where $G$ is a smooth algebraic group scheme, this descent data amounts to a $G$-equivariant structure on $F_Y$ (\cite[Def. I.3.46]{MR2222646}). 

Below we give a definition of the bounded derived category of coherent sheaves $D_{coh}^b(\X)$ on a stack $\X$. In the cases of interest to us this definition is equivalent to the one given in \cite{MR1771927}, but in the case of the unbounded derived category $D_{qcoh}(\X)$ of quasi-coherent sheaves we prefer to use a definition requiring slightly more machinery.

\begin{defi}
Let $\X$ be a quasi-compact algebraic stack with affine diagonal and atlas $Y \to X$. We define the bounded derived category $D^b_{coh}(\X)$ of coherent sheaves on $\X$ to be the full subcategory of the derived category of $\mathrm{QCoh}(\X)$ of complexes $F^{\bullet}$ whose cohomology sheaves are coherent when pulled back to $Y$ and vanish for almost all degrees.
\end{defi}

It is a well-known fact that the non-functoriality of cones leads to technical complications in the theory of derived categories. For instance, it is not possible to obtain $D_{qoh}(\X)$ as a $2$-limit of the derived categories $D_{qcoh}(U)$ for affine schemes $U \to \X$ as we did it for the abelian category above. And neither is the category of $G$-equivariant objects in $D_{qcoh}(Y)$ equivalent to the derived category of the quotient stack $[Y/G]$. This defect of $D_{qcoh}(\X)$ can be fixed by replacing the derived category by an \emph{enhancement}, i.e. a closely related object, from which $D_{qcoh}(\X)$ can be fully recovered, but which possesses a functorial construction of cones. One way to do this is by using the theory of stable $\infty$-categories \cite{Lurie:qf}. Every affine scheme $U$ has an associated stable $\infty$-category $QC(U)$, whose homotopy category is the derived category of quasi-coherent sheaves on $U$. For a prestack $\X$ one defines $QC(\X)$ as the homotopy limit of $\infty$-categories $$QC(\X) := \lim_{U \in \mathrm{Aff}/\X} QC(U),$$ in analogy with the definition of the category of quasi-coherent sheaves $\mathrm{QCoh}(\X)$ given at the beginning of this section.

\begin{defi}
Let $\X$ be an algebraic stack, the derived category of quasi-coherent sheaves $D_{qcoh}(\X)$ is defined to be the homotopy category of the stable $\infty$-category $QC(\X)$.
\end{defi}

Whenever possible we will formulate proofs in the language of derived categories, but complementing our presentation by using stable $\infty$-categories. The inherent functoriality in the language of stable $\infty$-categories allows straightforward constructions, which would be more intricate in the world of triangulated categories. We demonstrate this principle with an easy lemma, which lies at the heart of our treatment of the autoduality conjecture in the special cases considered here (Theorem \ref{thm:toyautoduality} and \ref{thm:autoduality}). A second proof, avoiding stable $\infty$-categories, will be supplied as Lemma \ref{lemma:equivariant}.

\begin{lemma}\label{lemma:equivariant_stable}
Let $X$ and $Y$ be two schemes, endowed with an action of an abstract finite group $\Gamma$; we assume that there is an equivalence of $\infty$-categories $$QC(X) \cong QC(Y),$$ which is $\Gamma$-equivariant. This induces an equivalence $$QC([X/\Gamma]) \cong QC([Y/\Gamma]).$$
\end{lemma}

\begin{proof}
Since $X \to [X/\Gamma]$ is an atlas for the stack $[X/\Gamma]$ it is possible to write $QC([X/\Gamma])$ as the homotopy limit $$\lim_{J \in \Delta^{op}} QC(X^{[J]}),$$ where $\Delta$ denotes the category of finite non-empty ordered sets and $$X^{[J]} := X \times \Gamma^J.$$ Let $B\Gamma$ be the nerve of the groupoid associated to the group $\Gamma$. The $\Gamma$-action on $X$ induces an action on $QC(X)$, which is encoded by an $\infty$-functor from $B\Gamma$ to the $\infty$-category of $\infty$-categories $$\mathsf{act}: B\Gamma \to \infty-\mathrm{Cat}.$$ The above homotopy limit can be rewritten as $$\lim_{B\Gamma} \mathsf{act},$$ which is a purely $\infty$-categorical construction, and therefore depends only on the $\infty$-category $QC(X)$ and the $\Gamma$-action up to equivalence. In general we refer to such a limit as the $\infty$-category of $\Gamma$-equivariant objects in a $\infty$-category. As the equivalence $QC(X) \cong QC(Y)$ respects the $\Gamma$-action, we obtain that the $\infty$-categories of $\Gamma$-equivariant objects in $QC(X)$ and $QC(Y)$ must be equivalent. In particular we have $$QC([X/\Gamma]) \cong QC([Y/\Gamma]).$$
\end{proof}

Even more generally, for an $\infty$-groupoid $G$, which is pointed and connected, and an $\infty$-functor $\mathsf{act}: G \to \infty-\mathrm{Cat},$ we should think of the homotopy limit $\C^{\Gamma} := \lim_{G} \mathsf{act}$ as an $\infty$-category of $G$-equivariant objects in an $\infty$-category $\C$. If $\C$ is stable (in particular its homotopy category is triangulated) then so is $\C^{\Gamma}$ according to Theorem 5.4 in \cite{Lurie:qf}. In \cite{Sosna:2011uq} an alternative \emph{linearization} procedure is described for triangulated categories having a strongly pre-triangulated dg-model. Using this definition of linearization P. Sosna also obtains an analogue of Lemma \ref{lemma:equivariant_stable} in \emph{loc. cit.}

\subsection{Fourier-Mukai transform}

Let $\X$, $\Y$ and $\Z$ be algebraic stacks, which we assume to be quasi-compact and having affine diagonal. Let $\X \to \Z$ and $\Y \to \Z$ be morphisms of stacks and $K \in D_{qcoh}(\X \times_{\Z} \Y)$ a complex on the fibre product $\X \times_{\Z} \Y$. If we denote by $p_{\X}: \X \times_{\Z} \Y \to \X$ the canonical projection, and similarly for $p_{\Y}$, we obtain an exact functor $$\Phi_{K}: D_{qcoh}(\X) \to D_{qcoh}(\Y),$$ which sends the complex of sheaves $F \in D_{qcoh}(\X)$ to $$\Phi_{K}(F) := Rp_{\Y,*}(Lp_{\X}^*F \otimes^L K).$$ Functor between derived categories of this type are referred to as (generalized) Fourier-Mukai transforms and were introduced by Mukai (\cite{MR946249}). The following statement is proved as in \cite{MR946249}, but using a slightly more general base change formula (e.g. \cite[Prop. 3.10]{MR2669705}). The proof can also be extracted from the proof of Lemma \ref{lemma:products}.

\begin{lemma}\label{lemma:composition}
Let $\X$, $\Y$ and $\Z$ be $\W$-stacks. We assume that all of these stacks are algebraic, quasi-compact and have affine diagonal; moreover we require that $\X \to \W$, $\Y \to \W$ and $\Z \to \W$ are representable flat morphisms. For $L \in D_{qcoh}(\X \times_{\W} \Y)$ and $K \in D_{qcoh}(\Y \times_{\W} \Z)$ we define $$L * K := Rp_{\X\Z,*}(Lp_{\X\Y}^*L \otimes^L Lp_{\Y\Z}^*K).$$ There exists a natural equivalence between the functors $\Phi_{K} \circ \Phi_L$ and $\Phi_{L * K}$.
\end{lemma}

As we are mainly dealing with generalized Fourier-Mukai functors, i.e. integral kernels living on a fibre product $\X \times_{\Z} \Y$, we have to investigate how the kernel changes if we replace the base $\Z$ along a morphism $\Z \to \W$. The behaviour of integral kernels under this change of base stack is expressed in the well-known lemma below, which is proved by application of the projection formula (e.g. \cite[Prop. 3.10]{MR2669705}).

\begin{lemma}\label{lemma:changeofbase}
Let $\X$, $\Y$, $\Z$ and $\W$ be algebraic stacks which are quasi-compact and have affine diagonal. We assume that $\X$ and $\Y$ are $\Z$-stacks, and that there is a schematic morphism $\Z \to \W$. Let $$f := (p_{\X},p_{\Y}): \X \times_{\Z} \Y \to \X \times_{\W} \Y$$ be the canonical morphism, and $K \in D_{qcoh}(\X \times_{\Z} \Y)$. Then the Fourier-Mukai transform $\Phi_K$ is naturally equivalent to $\Phi_{Rf_*K}$.
\end{lemma}

We will also have to understand the behaviour of Fourier-Mukai transform with respect to base change. 

\begin{lemma}\label{lemma:basechange}
Let $\X$, $\Y$, $\Z$ and $\W$ be perfect algebraic stacks which are quasi-compact and have affine diagonal. We assume that $\X$ and $\Y$ are flat $\Z$-stacks, and that there is a schematic morphism $\pi: \W \to \Z$. Every Fourier-Mukai equivalence $$\Phi_K: D_{qcoh}(\X) \cong D_{qcoh}(\Y)$$ relative to $\Z$ induces a Fourier-Mukai equivalence relative to $\W$ $$\Phi_{\pi^*K}: D_{qcoh}(\X \times_{\Z} \W) \cong D_{qcoh}(\Y \times_{\Z}\W)$$ by pulling back the kernel $K$.
\end{lemma}

\begin{proof}
Let $L \in D_{qcoh}(\X \times_{\Z} \Y)$ so that $K * L \cong \Delta_*\Oo_{\X}$ and $L * K \cong \Delta_*\Oo_{\Y}$. The base change formula implies that the same relations hold for $\X \times_{\Z} \W$ and $\Y \times_{\Z} \W$.
\end{proof}

The next lemma tells us that if $X_i$ and $Y_i$ are Fourier-Mukai partners for $i = 1,2$, then then products $X_1 \times X_2$ and $Y_1 \times Y_2$ are Fourier-Mukai partners. Using the formalism of stable $\infty$-categories, this is simply a consequence of Theorem 1.2 in \cite{MR2669705}.

\begin{lemma}\label{lemma:products}
For $i = 1,2$ let $\X_i$, $\Y_i$, $\mathcal{Z}_i$ and $\W$ be perfect algebraic stacks which are quasi-compact and have affine diagonal. We assume that $\X_i$, $\Y_i$, $\mathcal{Z}_i$ are $\W$-stacks and that the structural morphisms are flat. Let $D_{qcoh}(\X_i)\cong D_{qcoh}(\mathcal{Y}_i)$  be derived equivalences of Fourier-Mukai type, induced by integral kernels $K_i \in D_{qcoh}(\X_i \times_{\Z_i} \mathcal{Y}_i)$. Then $K_1 \boxtimes^L K_2$ induces a derived equivalence $$D_{qcoh}(\X_1 \times_{\W} \X_2) \cong D_{qcoh}(\mathcal{Y}_1 \times_{\W} \mathcal{Y}_2),$$ relative to $\Z_1 \times_{\mathcal{W}} \Z_2$.
\end{lemma}

\begin{proof}
According to Lemma \ref{lemma:basechange} we know that the equivalences $D_{qcoh}(\X_i) \cong D_{qcoh}(\Y_i)$ induce equivalences $$D_{qcoh}(\X_1 \times_{\W} \Y_1) \cong D_{qcoh}(\X_2 \times_{\W} \Y_1)$$ and $$D_{qcoh}(\X_2 \times_{\W} \Y_1) \cong D_{qcoh}(X_2 \times_{\W} Y_2).$$ By juxtaposition we obtain a derived equivalence $$D_{qcoh}(\X_1 \times_{\W} \Y_1) \cong D_{qcoh}(\X_2 \times_{\W} \Y_2).$$ In order to obtain a better understanding of the integral kernel of this composition we take a look at the following commutative diagram with cartesian squares:
\begin{center}
\begin{tikzpicture}

\node at (0,1) {$\X_1 \times_{\W} \X_2$};
\node at (2.5,0) {$\Z_1 \times_{\W} \X_2$};
\node at (5,1) {$\Y_1 \times_{\W} \X_2$};
\node at (2.5,2) {$(\X_1 \times_{\Z_1}\Y_1) \times_{\W} \X_2$};

\node at (7.5,0) {$\Y_1 \times_{\W} \Z_2$};
\node at (10,1) {$\Y_1 \times_{\W} \Y_2$};
\node at (7.5,2) {$\Y_1 \times_{\W} (\X_2 \times_{\Z_2}\Y_2)$};

\node at (5,3) {$(\X_1 \times_{\W} \X_2) \times_{(\Z_1 \times_{\W} \Z_2)} (\Y_1 \times_{\W} \Y_2)$};

\node at (0.9,1.7) {$p$};
\node at (4.1,1.7) {$q$};
\node at (5.9,1.7) {$r$};
\node at (9.1,1.7) {$s$};

\node at (3.4,2.7) {$\alpha$};
\node at (6.6,2.7) {$\beta$};

\draw [->,thick] (0.3,0.7)--(2.2,0.3);
\draw [->,thick] (4.7,0.7)--(2.8,0.3);
\draw [->,thick] (5.3,0.7)--(7.2,0.3);
\draw [->,thick] (9.7,0.7)--(7.8,0.3);

\draw [->,thick] (2.2,1.7)--(0.3,1.3);
\draw [->,thick] (2.8,1.7) -- (4.7,1.3);
\draw [->,thick] (7.2,1.7)--(5.3,1.3);
\draw [->,thick] (7.8,1.7)--(9.7,1.3);

\draw [->,thick] (4.7,2.7)--(2.8,2.3);
\draw [->,thick] (5.3,2.7)--(7.2,2.3);
\end{tikzpicture}
\end{center}
Let $M \in D_{qcoh}(\X_1 \times \Y_1)$, we denote by $$c: (\X_1 \times_{\Z_1}\Y_1) \times \X_2 \to \X_1 \times_{\Z_1}\Y_1$$ and $$d: \Y_1 \times (\X_2 \times_{\Z_2}\Y_2) \to \X_2 \times_{\Z_2}\Y_2$$ the canonical projection; the base change formula reveals now that $$Rs_*((Lr^*Rq_*(Lp^*\otimes^L Lc^*K_1))\otimes^L Ld^*K_2) \cong Rs_*(R\beta_*(L\alpha^*Lp^*M \otimes^L K_1) \otimes^L Ld^*K_2).$$ Using the projection formula we obtain $$Rs_*R\beta_*(L\alpha^*Lp^*M \otimes L\alpha^*Lc^*K_1 \otimes^L L\beta^*Ld^*K_2).$$ In particular we see that the integral kernel is given by $K_1 \boxtimes^L K_2$.
\end{proof}

The lemma below is reminiscent from Lemma \ref{lemma:equivariant_stable}; it is formulated and proved in a more classical language, but using more restrictive assumptions.

\begin{lemma}\label{lemma:equivariant}
Let $X$, $Y$ and $Z$ be quasi-projective smooth $k$-varieties, proper and flat over $Z$, endowed with the action of an abstract finite group $\Gamma$, such that the characteristic of $k$ does not divide $\Gamma$; we assume that there is a functor of Fourier-Mukai type $$\Phi_K: D^b_{coh}(X) \to D^b_{coh}(Y),$$ given by an integral kernel $K \in D_{qcoh}(X \times_Z Y)$ of finite Tor-dimension. Moreover we assume that $K$ is endowed with a $\Gamma$-equivariant structure; in the sense that $K \cong f^*L$ for $$L \in D^b_{coh}([X/\Gamma] \times_{[Z/\Gamma]} [Y/\Gamma]),$$ and the obvious map $$f: X \times_Z Y \to [X/\Gamma] \times_{[Z/\Gamma]} [Y/\Gamma].$$ Then $\Phi_K$ is an equivalences of categories if and only if $\Phi_L$ is.
\end{lemma}

\begin{proof}
Let $W$ be a scheme with a $\Gamma$-action. We denote by $$f_W: W \to [W/\Gamma]$$ the canonical morphism to the quotient stack. All these maps are faithfully flat and therefore we do not have to distinguish between $f_W^*$ and $Lf_W^*$. For $L \in D^b_{coh}([X/\Gamma] \times_{[Z/\Gamma]} [Y/\Gamma])$ we denote the associated Fourier-Mukai transform by  $\Phi_{[X/\Gamma][Y/\Gamma]}$, and similarly will $\Phi_{[X/\Gamma]Y}$ denote the functor $D^b_{coh}([X/\Gamma]) \to D^b_{coh}(Y)$ associated to the kernel obtained from $L$ via pullback along the obvious map, and so on.

The base change formula implies $$f_Y^*\Phi_{[X/\Gamma][Y/\Gamma]} \cong \Phi_{[X/\Gamma]Y} \cong \Phi_{XY}f_X^*.$$ Since $f_Y$ and $f_X$ are faithfully flat, we conclude that $f_Y^*$ and $f_X^*$ are fully faithful functors. Together with $\Phi_{XY}$ being fully faithful we conclude that $\Phi_{[X/\Gamma][Y/\Gamma]}$ is fully faithful.

The functor $\Phi_{[X/\Gamma][Y/\Gamma]}$ has a right adjoint $\Psi_{[X/\Gamma][Y/\Gamma]}$ given by the integral kernel $L^{\dual}\otimes \omega$ on $[Y/\Gamma]\times_{[Z/\Gamma]}[X/\Gamma]$, where $\omega$ is a factor derived from a relative dualizing line bundle. This is seen by using the standard adjunctions together with Grothendieck-Serre duality. Replacing $\Phi$ by $\Psi$ above, we see that $$f_X^*\Psi_{[Y/\Gamma][X/\Gamma]}(F) \cong \Psi_{YX}(f_Y^*F),$$ for $F \in D^b_{coh}([Y/\Gamma])$. Since $\Psi_{YX}$ is a quasi-inverse to $\Phi_{XY}$, we obtain from this relation that $\Psi_{[Y/\Gamma][X/\Gamma]}(F) = 0$ if and only if $F = 0$, again by virtue of the fact that $f_X$ and $f_Y$ are faithfully flat. Lemma 2.1 of \cite{Bridgeland:fk} implies now that $\Phi_{[X/\Gamma][Y/\Gamma]}$ is an equivalence of categories.
\end{proof}

We emphasize that every functor of Fourier-Mukai type $D_{qcoh}(\X) \to D_{qcoh}(\Y)$ lifts to an $\infty$-functor $QC(\X) \to QC(\Y)$. And for a large class of stacks every functor between $\infty$-categories $QC(\X) \to QC(\Y)$ is obtained from a Fourier-Mukai transform (\cite[Thm. 1.2(2)]{MR2669705}). For our purposes it is therefore merely a matter of taste whether one utilizes the theory of stable $\infty$-categories or derived categories and functors of Fourier-Mukai type.

\subsection{The McKay correspondence}

In the paper \cite{Bridgeland:fk} an important special case of the Fourier-Mukai transform has been considered to establish a form of the derived McKay correspondence. We denote by $X$ a smooth quasi-projective variety with an abstract finite group $\Gamma$ acting on it. Moreover we assume that the characteristic of $k$ is zero or $p > |\Gamma|$. Let us denote by $Y \subset \mathrm{Hilb}^{|\Gamma|} [X/\Gamma]$ the scheme representing the functor given by $\Gamma$-equivariant subschemes $$Z \to X,$$ such that there exists a surjection of $\Gamma$-equivariant sheaves $\Oo_{X} \twoheadrightarrow \Oo_Z$ and $\Gamma$ acts on $H^0(Z,\Oo_Z)$ as the regular representation. Moreover we remove redundant irreducible components, so that we are left with the irreducible component containing the free $\Gamma$-orbits.

The fibre product $Y \times [X/\Gamma]$ is endowed with the structure sheaf $\Oo_{\mathcal{Z}}$ of the universal $\Gamma$-cluster $\mathcal{Z}$.

\begin{thm}[{\cite[Thm. 1.1]{Bridgeland:fk}}]\label{thm:BKR}
We assume that the $\Gamma$-Hilbert scheme $Y$ of $X$ is smooth and satisfies the estimate $$\dim Y \times_{X/\Gamma} Y \leq \dim Y+1,$$ where $X/\Gamma$ denotes the GIT quotient. Then the structure sheaf of the universal family $\Oo_\mathcal{Z}$ of $\Gamma$-\emph{clusters} on $$Y \times [X/\Gamma]$$ induces an equivalence of $k$-linear derived categories of Fourier-Mukai type $$D^b_{coh}(Y) \cong D^b_{coh}([X/\Gamma]).$$
\end{thm}

In \cite{Bridgeland:fk} a slightly more general Theorem is proved for $k = \mathbb{C}$. The reason for this restriction is the use of the so-called \emph{New Intersection Theorem} due to Roberts (\cite{roberts}) and Peskine--Szpiro, which guarantees smoothness of $Y$. While this Theorem holds in positive characteristic, \cite{Bridgeland:fk} uses an addendum proved in \cite{bi} in the right generality. Nonetheless, in cases of interest to us, $Y$ will be already known to be smooth for different reasons.

In the next lemma we observe how this Fourier-Mukai transform interacts with a given morphism $X \to S$. This will be important for our analysis of autoduality of the Hitchin fibration in Theorem \ref{thm:toyautoduality} and $\ref{thm:autoduality}$.

\begin{lemma}\label{lemma:linear}
Let $\pi: X \to S$ be a flat morphism of smooth quasi-projective varieties, endowed with the action of a finite group $\Gamma$, such that $\pi$ is $\Gamma$-equivariant. If $X$ satisfies the conditions of Theorem \ref{thm:BKR}, $Y$ denotes the $\Gamma$-equivariant Hilbert scheme as before and $S/\Gamma$ the GIT quotient, then the natural equivalence of derived categories $$D^b_{coh}(Y) \cong D^b_{coh}([X/\Gamma])$$ is of Fourier-Mukai type relative to $S/\Gamma$.
\end{lemma}

\begin{proof}
We only need to check that $\Oo_{\mathcal{Z}}$ is supported on the fibre product $$Y \times_{S/\Gamma} [X/\Gamma],$$ which one expects to be a consequence of the $\Gamma$-equivariance of $\pi$. To show this we may cover $S/\Gamma$ by Zariski open affine subsets $U_i$ and cover $S$ by the fibre products $$S_i := U_i \times_{S/\Gamma}S,$$ which are still affine, as $S \to S/\Gamma$ is finite. Using quasi-projectivity of $X$ we can cover $X\times_S S_i \subset X$ by Zariski open affine subsets $V_i$, which are $\Gamma$-invariant. Henceforth we may assume without loss of generality that $X = \Spec A$ and $S = \Spec D$ are affine varieties endowed with the action of an abstract group $\Gamma$.

Let now $C$ be another algebra endowed with the trivial $\Gamma$-action and $B$ a $C$-flat quotient of $A \otimes C$ sitting in a short exact sequence
\begin{center}
\begin{tikzpicture}

\node at (0,0) {$0$};
\node at (1,0) {$I$};
\node at (2.5,0) {$A \otimes C$};
\node at (4,0) {$B$};
\node at (5,0) {$0$};

\draw [->,thick] (0.3, 0) -- (0.7,0);
\draw [->,thick] (1.4, 0) --(1.8,0);
\draw [->,thick] (3.2,0) -- (3.6, 0);
\draw [->,thick] (4.3,0) -- (4.7,0);

\end{tikzpicture}
\end{center}
such that $I$ is a $\Gamma$-invariant ideal of $A \otimes C$. In particular this is a short exact sequence of $\Gamma$-modules. Moreover we assume that $C \to A \otimes C \to B$ induces an isomorphism $$C \cong B^{\Gamma}.$$ This algebraic data encodes a $C$-point of the $\Gamma$-Hilbert scheme of $X$. 

Since $A$ is a $D$-algebra we obtain a natural morphism $$D^{\Gamma} \to B^{\Gamma} = C,$$ endowing $C$ with the structure of a $D^{\Gamma}$-algebra. Therefore we see that on the $A \otimes C$-module $B$ the action of $D^{\Gamma}$ via $A$ agrees with the action via $C$. Thus $B$ is actually a $A \otimes_{D^{\Gamma}} C$-module, which is what we wanted to show.
\end{proof}

An important example is given by Hilbert schemes of surfaces. To see how those relate to equivariant Hilbert schemes we quote the following result of M. Haiman (\cite[Thm. 6]{MR1839919}), which is a corollary of his proof of the $n!$-conjecture.\footnote{In \cite{MR1839919} this Theorem is stated for $k=\mathbb{C}$. The proof translates without problems to the more general case of characteristic zero or $p>n$. This is due to the fact that the main technical ingredient of Haiman's proof, the Polygraph Theorem, is already established in the required generality.}

\begin{thm}[Haiman]
Let $X$ be a surface defined over a field of characteristic $p > n$ or zero. Let us denote by $X^{[n]}$ the Hilbert scheme of length $n$ subschemes and by $Y_n$ the $S_n$-Hilbert scheme of $X^n$ with respect to the natural group action of the symmetric group $S_n$ given by permuting factors. Then there is a natural isomorphism $$Y_n \cong X^{[n]}.$$  
\end{thm}

Combining this result with Theorem \ref{thm:BKR} we obtain a well-known derived equivalence.

\begin{cor}\label{cor:Hilbert}
If $X$ denotes a surface defined over a field of characteristic $p > n$ or zero, and $X^{[n]}$ denotes the Hilbert scheme of length $n$ subschemes, then we have a natural derived equivalence $$D^b_{coh}(X^{[n]}) \cong D^b_{coh}([X^n/S_n]).$$
\end{cor}

Note that the required dimension estimate follows from the classical result of Brian\c{c}on (\cite{MR0457432}) and Iarrabino (Corollary 1 in \cite{MR0308120}) that for the punctual Hilbert scheme $\Hilb^m_0 \mathbb{A}^2$ we have $$\dim \Hilb^m_0 \mathbb{A}^2 = m-1.$$ 

\section{Moduli spaces of dimension two}\label{sec:surficial}

Let $Q$ be a Dynkin diagram, such that the corresponding affine Dynkin diagram $\tilde{Q}$ is comet-shaped (see subsection \ref{ssec:parabolic}). The only Dynkin diagrams satisfying this assumption are $A_0$, $D_4$, $E_6$, $E_7$ and $E_8$. In subsection \ref{ssec:parHiggs} we explained how comet-shaped graphs together with a dimension vector encode moduli problems for parabolic Higgs bundles. The graphs $\tilde{Q}$ listed above together with the basic imaginary root $\lambda$ are exactly the ones corresponding to the moduli spaces of parabolic Higgs bundles of dimension $2$ which we will describe.

The $A_0$-case is the simplest one, it describes Higgs bundles of rank one on an elliptic curve $E$ and the moduli space is $T^*E$. Nonetheless there are many other examples of two-dimensional moduli spaces of Higgs bundles that are somehow reminiscent of this one. To each graph describing a moduli space of parabolic Higgs bundles we can associate a finite group $\Gamma$. For $D_4$, $E_6$, $E_7$ and $E_8$ these groups are $\mathbb{Z}/2\mathbb{Z}$, $\mathbb{Z}/3\mathbb{Z}$, $\mathbb{Z}/4\mathbb{Z}$ and $\mathbb{Z}/6\mathbb{Z}$. Let now $E$ be an appropriate elliptic curve with a $\Gamma$-action. In the $D_4$-case $E$ is an arbitrary elliptic curve with $\mathbb{Z}/2\mathbb{Z}$ acting on it via $x \mapsto -x$. In all the other cases the $\Gamma$-action stems from complex multiplication on the curve $E$, as has been explained in the introduction. This allows us to formulate the following folklore Theorem which will be proved in subsection \ref{ssec:crepant}.

\begin{thm}\label{thm:toyexamples}
The moduli space of stable parabolic Higgs bundles $\ssHiggs(Q,{\lambda})$ of orbifold degree zero, with respect to weights $\alpha_i := \frac{i}{n}$ for $i < n-1$ and $1 > \alpha_{n-1} > \frac{n-1}{n}$, associated to the orbifold $[E/\Gamma]$ is naturally isomorphic to the $\Gamma$-Hilbert scheme of the surface $T^*E$.
\end{thm}

A formula for the dimension of $\ssHiggs(\widehat{X},n,\lambda_{\bullet})$ is given in \cite[p. 3]{MR1411302}, assuming the moduli space is non-empty: 
\begin{equation}\label{eqn:dimension}
2(g-1)n^2 + 2 + \sum_{p \in D} (n^2 - \sum_{i = 1}^{n_p} (\lambda_{pi +1} - \lambda_{pi})^2).
\end{equation} 
We have the estimate $$\sum_{i = 1}^{n_p} (\lambda_{pi +1} - \lambda_{pi})^2 \leq n^2$$ which follows from the inequality 
\begin{equation}\label{eqn:inequality}
\sum_{i = 1}^nx_i^2 \leq (\sum_{i=1}^nx_i)^2,
\end{equation} 
where $x_i \geq 0$. In particular we see that there are two possible cases, where the expression (\ref{eqn:dimension}) specializes to $2$. If $g=1$ and $D=0$, since the inequality (\ref{eqn:inequality}) is strict if there are two non-zero summands; and $g = 0$ and $$-2 n^2 + \sum_{p \in D} (n^2 - \sum_{i = 1}^{n_p-1} (\lambda_{pi +1} - \lambda_{pi})^2) = 0.$$ This expression on the other hand is $-2q$, where $q$ denotes the quadratic form associated to the star-shaped graph

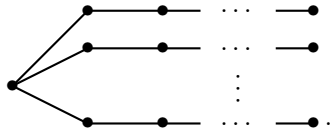
\begin{center}
\begin{tikzpicture}
\node at (2,0) {$\bullet$};
\node at (3,1) {$\bullet$};
\node at (4,1) {$\bullet$};
\node at (5,1) {$\dots$};
\node at (6,1) {$\bullet$};
\draw [thick] (2,0) -- (3,1);
\draw [thick] (3,1) -- (4,1);
\draw [thick] (4,1) -- (4.5,1);
\draw [thick] (5.5,1) -- (6,1);

\node at (3,0.5) {$\bullet$};
\node at (4,0.5) {$\bullet$};
\node at (5,0.5) {$\dots$};
\node at (6,0.5) {$\bullet$};
\draw [thick] (2,0) -- (3,0.5);
\draw [thick] (3,0.5) -- (4,0.5);
\draw [thick] (4,0.5) -- (4.5,0.5);
\draw [thick] (5.5,0.5) -- (6,0.5);

\node at (5,0) {\rotatebox{90}{$\dots$}};

\node at (3,-0.5) {$\bullet$};
\node at (4,-0.5) {$\bullet$};
\node at (5,-0.5) {$\dots$};
\node at (6,-0.5) {$\bullet$};
\node at (6.2,-0.5) {.};

\draw [thick] (2,0) -- (3,-0.5);
\draw [thick] (3,-0.5) -- (4,-0.5);
\draw [thick] (4,-0.5) -- (4.5,-0.5);
\draw [thick] (5.5,-0.5) -- (6,-0.5);
\end{tikzpicture}
\end{center}

If $Q$ is of affine Dynkin type, we see in particular that all such dimension vectors are multiples of the basic imaginary root $\alpha$.

\subsection{Duality for elliptic curves with symmetries}

Let $A$ be an abelian variety, the dual abelian variety $A^{\dual}$ is equivalent to the stack $Map_{grp}(A,B\G_m)$ representing morphisms of group stacks $A \to B\G_m$ (\cite[p. 184]{MR918564}). Equivalently we can say that $A^{\dual}$ classifies extensions of $A$ by $\G_m$. This construction is analogous to the dual of a vector space $V^{\dual} := Hom(V,k)$, with $B\G_m$ taking the place of the one-dimensional vector space $k$. For the same reason as there is a canonical morphism $V \to V^{\dual\dual}$ for vector spaces there is a a canonical morphism $$\psi_A: A \to A^{\dual\dual},$$ which is an isomorphism. This in turn gives rise to a morphism $$A \times A^{\dual} \to B\G_m.$$
As the stack $B\G_m$ classifies line bundles, we see that there is a canonical line bundle $\Pp$ on $A \times A^{\dual}$, called \emph{Poincar\'e bundle}. There is a general duality theory for group stacks, an exposition of which is given in \cite{AriApp}.

It has been shown by Mukai in \cite{MR946249} that the Poincar\'e line bundle $\Pp$ induces a natural equivalence of categories 
\begin{equation}\label{eqn:mukai}
D^b_{coh}(A) \cong D^b_{coh}(A^{\dual}).
\end{equation}
If $\phi: A \to B$ is a morphism we obtain a dual morphism $\phi^{\dual}: B^{\dual} \to A^{\dual}$ which sends an $S$-point $f: B \times S \to B\G_m \times S$ of $B^{\dual}$ to the composition 
\begin{center}
\begin{tikzpicture}
\node at (0,0) {$A \times S$};
\node at (2.6,0) {$B \times S$};
\node at (5.5,0) {$B\G_m \times S.$};

\node at (1.3,0.3) {$\phi \times \id_S$};
\node at (4,0.3) {$f$};

\draw [->,thick] (0.6, 0) -- (2,0);
\draw [->,thick] (3.2,0) -- (4.6,0);
\end{tikzpicture}
\end{center}

By definition the diagram
\begin{center}
\begin{tikzpicture}
\node at (0,1.5) {$A$};
\node at (0,0) {$B$};
\node at (1.7,1.5) {$A^{\dual\dual}$};
\node at (1.7,0) {$B^{\dual\dual}$};

\node at (-0.2,0.75) {$\phi$};
\node at (1.9,0.75) {$\phi^{\dual\dual}$};

\draw [->,thick] (0.3,0) -- (1.2,0);
\draw [->,thick] (0.3,1.5) -- (1.2,1.5);
\draw [->,thick] (0, 1.2) -- (0, 0.3);
\draw [->,thick] (1.5, 1.2) -- (1.5, 0.3);

\end{tikzpicture}
\end{center}
is commutative. In particular we conclude that if $\Gamma$ is a finite group acting on $A$ then Mukai's equivalence (\ref{eqn:mukai}) is $\Gamma$-equivariant; in the strong sense that the integral kernel $\Pp$ is endowed with a $\Gamma$-equivariant structure.

If $E$ denotes an elliptic curve we may identify $$T^*E = E \times \A = \Higgs(E^{\dual},1),$$ where $\A$ denotes the Hitchin base. Note that there is a canonical identification of elliptic curves $E \cong E^{\dual}$, given by the Abel-Jacobi map. This induces an identification of Hitchin bases $$\A(E) = \A(E^{\dual}).$$ Using this autoequivalence, the remarks above, Lemma \ref{lemma:changeofbase} and Lemma \ref{lemma:equivariant} we arrive at the following well-known observation.

\begin{prop}\label{prop:equivariant}
There is a canonical equivalence of derived categories of Fourier-Mukai type relative to $\A$ $$D^b_{coh}(T^*E) \cong D^b_{coh}(T^*E^{\dual}).$$ If $E$ is equipped with a $\Gamma$-action this equivalence respects the $\Gamma$-action, in particular we have an equivalence of derived categories of Fourier-Mukai type relative to $\A/\Gamma$  $$D^b_{coh}([T^*E/\Gamma]) \cong D^b_{coh}([T^*E^{\dual}/\Gamma]).$$
\end{prop}

\subsection{Higgs bundles and crepant resolutions}\label{ssec:crepant}

In this subsection we show how moduli spaces of Higgs bundles give rise to crepant resolutions of certain quotients of cotangent bundles of elliptic curves. We start with a technical definition, which turns out to be essential for relating Higgs bundles with torsion sheaves via Fourier-Mukai transform.

\begin{defi}\label{defi:admissible}
Let $X$ be an orbicurve and $(E,\theta)$ a Higgs bundle on it. A composition series for $(E,\theta)$ is an increasing filtration by Higgs subbundles $$(E^{\bullet},\theta) \subset (E,\theta),$$ such that the successive quotients $E^{i+1}/E^i$, called factors, are locally free and have no non-trivial Higgs bundle as a quotient. The Higgs bundle $(E,\theta)$ is said to be \emph{admissible} if there exists a composition series, such that all factors are of rank one and degree zero. An $S$-family of Higgs bundles is called admissible if it is admissible over every geometric point of $S$. We denote the stack of rank $n$ admissible Higgs bundles on an orbicurve $X$ by $\adHiggs(X,n)$.
\end{defi}

It is a well-known fact (Lemma 4.2(1) in \cite{MR2484736}) that an extension of semistable objects of the same slope is again semistable. We record the following implication for admissible Higgs bundles for later use.

\begin{rmk}\label{rmk:admissible}
Admissible Higgs bundles are semistable of slope zero.
\end{rmk}

We will see later that admissible Higgs bundles correspond to torsion sheaves supported on a dual Hitchin fibration. This will induce an equivalence of stacks. 

\begin{defi}
For an orbisurface $S$ we denote by $\Tt(S,n)$ the stack of length $n$ torsion free sheaves on $S$, i.e. the $2$-functor $$\mathrm{Aff}^{op} \to Grpd$$ which sends an affine scheme $T$ to the groupoid of quasi-coherent sheaves $\F$ on $S \times T$, such that $\pi: \supp \F \to T$ is finite and $\pi_*\F$ is locally free of rank $n$ on $S$.
\end{defi}

In the next lemma we formulate how admissible Higgs bundles are related to torsion sheaves.

\begin{lemma}\label{lemma:admissible}
The equivalence of Proposition \ref{prop:equivariant} gives rise to an equivalence of stacks $$\Tt(T^*E,n) \cong \adHiggs(E^{\dual},n).$$
\end{lemma}

\begin{proof}
Given a $T$-point $\F$ of $\Tt(T^*E,n)$ we have to verify that the Fourier-Mukai transform $\Phi(\F)$ on $T^*E^{\dual}$ is a $T$-family of quasi-coherent sheaves on $T^*E^{\dual}$, i.e. a Higgs bundle on $E^{\dual}$ via the BNR correspondence (Prop. \ref{prop:BNR}). This formulation is justified, as we know from Lemma \ref{lemma:basechange} that for every $k$-scheme $T$ there is an induced Fourier-Mukai transform $$\Phi: D^b_{coh}(T^*E \times T) \cong D^b_{coh}(T^*E^{\dual} \times T).$$

If $$\pi: T^*E^{\dual} \to E^{\dual}$$ denotes the canonical projection, we need to verify that $\pi_*\Phi(\F)$ is a locally free sheaf of rank $n$. This push-forward can be calculated as the Fourier-Mukai transform of $\F$ along the functor $$\Psi: D^b_{coh}(T^*E\times T) \to D^b_{coh}(E^{\dual}\times T)$$ induced by the Poincar\'e bundle $\Pp$ on $E \times E^{\dual}$. Let $p_1: T^*E \times E^{\dual} \to T^*E$ and $p_2: T^*E \times E^{\dual} \to E^{\dual}$ denote the canonical projections onto the factors, respectively their base changes with respect to $T$. Then we have$$\pi_*\Phi(\F) = \Psi(F) = p_{2,*}(p_1^*\F \otimes \Pp).$$ But since $\supp \F \to T$ is finite and $\Pp$ is a line bundle, we see that this is a locally free sheaf of rank $n$ on $E^{\dual}$. A similar Fourier-Mukai set-up was used in \cite{MR1859601} to define the vector bundle underlying the Higgs bundle constructed from a torsion sheaf.

We also need to check that the Fourier-Mukai transform $\Phi(\F)$ is a family of admissible Higgs bundles on $[E/\Gamma]$. For this we may replace $S$ by a geometric point and therefore assume that $\F$ has a composition series $\F^{\bullet}$, such that the successive quotients $\F^{i+1}/\F^i$ are skyscraper sheaves of length one. This composition series can be encoded in a sequence of distinguished triangles
\begin{center}
\begin{tikzpicture}
\node at (0,0.05) {$\F^i$};
\node at (1.5,0.05) {$\F^{i+1}$};
\node at (3.5,0.05) {$\F^{i+1}/\F^{i}$};
\node at (5.05,-0.14) {.};
\node at (4.6, 0.15) {$\bullet$};

\draw [->,thick] (0.3,0) -- (1,0);
\draw [->,thick] (1.9,0) -- (2.6,0);
\draw [->,thick] (4.3,0) -- (5,0);
\end{tikzpicture}
\end{center}
Applying the equivalence $\Phi$ to $\F$ we see that $\Phi(\F)$ may be filtered by distinguished triangles
\begin{center}
\begin{tikzpicture}
\node at (0,0.05) {$\Phi(\F^i)$};
\node at (1.5,0.05) {$\Phi(\F^{i+1})$};
\node at (3.5,0.05) {$\Phi(\F^{i+1}/\F^{i})$};
\node at (4.9,-0.14) {.};
\node at (4.6, 0.15) {$\bullet$};

\draw [->,thick] (0.5,0) -- (0.8,0);
\draw [->,thick] (2.2,0) -- (2.5,0);
\draw [->,thick] (4.5,0) -- (4.8,0);
\end{tikzpicture}
\end{center}
By assumption $\Phi(\F^{i+1}/\F^i)$ is a quasi-coherent sheaf $T^*E^{\dual}$, corresponding to a rank one degree zero Higgs bundle on $E^{\dual}$ via the BNR correspondence (Prop. \ref{prop:BNR}). By induction on $n$ we obtain that $\Phi(\F)$ corresponds to an admissible Higgs bundle.

Similarly we see that an admissible Higgs bundle of rank $n$ on $E^{\dual}$ is sent to a length $n$ torsion free sheaf on $T^*E$.
\end{proof}

We have found a way of relating torsion sheaves on the surface $T^*E$ to Higgs bundles on the dual elliptic curve $E^{\dual}$. As a next step we investigate the transform of a point of the $\Gamma$-Hilbert scheme $Y$ of $T^*E$. Such a point gives rise to a $\Gamma$-equivariant torsion sheaf $\F$ on $T^*E$ together with a $\Gamma$-equivariant surjection $s: \Oo_{T^*E} \twoheadrightarrow \F$. As a first approximation we expect to obtain a $\Gamma$-equivariant Higgs bundle of rank $|\Gamma|$ on $E^{\dual}$, due to the functoriality of the construction described above. In the proof below we investigate the structure corresponding to the surjection $s$.

\begin{proof}[Proof of Theorem \ref{thm:toyexamples}]
The $\Gamma$-Hilbert scheme $Y$ of $T^*E$ can be defined in terms of $\mathcal{T} = \mathcal{T}(T^*[E/\Gamma])$. An $S$-point of $Y$ consists of an $S$-point $\F$ of $\mathcal{T}$ together with a surjection $$s: \Oo_{[T^*E/\Gamma] \times S} \twoheadrightarrow \F.$$ Moreover we demand that the $\Gamma$-representation $$Hom(\Oo_{[T^*E/\Gamma]\times S}, \F)$$ is the regular $S$-linear representation of $\Gamma$. We can now try to understand how $s$ transforms under the equivalence of categories $\Phi$.

Let us denote by $T_0^*E$ the closed subscheme of $T^*E$ given by the fibre over zero of $T^*E \to E$. The equivalence $D^b_{coh}(T^*E) \cong D^b_{coh}(T^*E^{\dual})$ sends $\Oo_{T^*E}$ to $\Oo_{T_0^*E^{\dual}}[-1].$ In particular we see that $v:=\Phi(s)$ is a morphism $$v: \Oo_{T_0^*E^{\dual}}[-1] \to \Phi(\F).$$ Serre duality tells us that this is equivalent to a morphism $$v': (\Phi(\F)_0)^{\dual} \to k,$$ i.e. an element $v \in \Phi(\F)_0$. Under this equivalence a morphism $\Oo_{\M} \to \F$ corresponds to a linear map $\tau_*(V^{\dual} \otimes \Oo/L_0^{-1}) \to k$, where $(V,\theta)$ is the $\Gamma$-equivariant Higgs bundle associated to $\F \in \Tt$. But $$\tau_*(V^{\dual} \otimes \Oo/L_0^{-1}) \cong E^{\dual}_0/F^{\dual}_1. $$ In particular we obtain a non-trivial linear map $$k \to F_{n_0-1}$$ by dualizing, i.e. a nonzero vector $v \in F_{n_0-1}$.

The vector space $Hom(\Oo_{[T^*E/\Gamma]\times S}, \F)$ corresponds to $\Phi(\F)_0$, as the argument given above tell us. In particular we see that $Y$ is equivalent to the moduli stack of the data $$(\Ee,\theta,v),$$ where $(\Ee,\theta)$ is an admissible Higgs bundle on $[T^*E/\Gamma]$, $\Ee_0$ is the regular $\Gamma$-representation and $v \in \Ee_0^{\Gamma}$ is a non-zero vector spanning the invariant part of $\Ee_0$. The latter is naturally equivalent to the moduli space of admissible Higgs bundles on $[T^*E/\Gamma]$, such that $\Ee_0$ carries the regular representation. Now we may apply Lemma \ref{lemma:isotropy-representation} to see that this corresponds exactly to the required type of parabolic bundles. 

Stability of the parabolic Higgs bundles follows from the fact that all $\Gamma$-invariant subbundles are of orbifold slope $\leq 0$ (see Remark \ref{rmk:admissible}) and that $\Phi(\F)$ is the only degree zero subbundle containing $v$. Since the weights are the canonical weights except from $\alpha_{n-1}$, stability follows. 

Note that $Y$ is naturally a $\A/\Gamma$-space with respect to the structural morphism $$Y \to T^*E/\Gamma \to \A/\Gamma.$$ Here the first morphism is the Hilbert-Chow morphism. The fact that $\Phi: D^b_{coh}([T^*E/\Gamma]) \cong D^b_{coh}([T^*E^{\dual}/\Gamma]$ is defined relative to $\A/\Gamma$ implies that the morphism $Y \to \M(\widehat{[E^{\dual}/\Gamma]},Q,\lambda)$ is a morphism of $\A/\Gamma$-spaces. We observe as well that this map is proper.

Therefore we have a morphism $$Y \to \M(\widehat{[E^{\dual}/\Gamma]},Q,\lambda)$$ of proper $\A/\Gamma$-spaces. Since both spaces are of equal dimension and connected we conclude that it is surjective. In particular we obtain that every Higgs bundle in the moduli space $\M(\widehat{[E^{\dual}/\Gamma]},Q,\lambda)$ is admissible. This allows us to conclude that $$Y \cong \M(\widehat{[E^{\dual}/\Gamma]},Q,\lambda),$$ with the inverse map provided by the inverse of the Fourier-Mukai transform $\Phi^{-1}$.
\end{proof}

\subsection{Local systems and crepant resolutions}

Using the categorification of geometric class field theory obtained in \cite{Laumon:ys} and \cite{MR1408538} we are able to prove the analogous result for moduli spaces of local systems by similar techniques. In the following we denote by $D_{qcoh}(X,D_X)$ the derived category of quasi-coherent $D_X$-modules on a smooth variety $X$. 

\begin{thm}[Laumon \& Rothstein]\label{thm:laumon-rothstein}
If $A$ is an abelian variety defined over an algebraically closed field of characteristic zero we denote by $A^{\sharp}$ the moduli space of local systems on $A$. Then there exists a canonical equivalence of derived categories
$$\Phi_{CFT}: D_{qcoh}(A^{\sharp}) \cong D_{qcoh}(A^{\dual},D_{A^{\dual}}).$$
\end{thm}

Note that in positive characteristic we define the ring of differential operators $D_X$ of a smooth variety to be the universal enveloping algebra of the Lie algebroid of tangent vectors $\Theta_X$. The analogue of the above Theorem in positive characteristic is proved in \cite[Cor. 3.8]{chenzhu}, using the techniques developed in \cite{Bezrukavnikov:fr} and \cite{Bezrukavnikov:mz}.

\begin{thm}[Chen--Zhu]\label{thm:chen-zhu}
If $A$ is an abelian variety defined over an algebraically closed field of positive characteristic we denote by $A^{\sharp}$ the moduli space of local systems on $A$. Then there is a canonical equivalence of derived categories $$D_{qcoh}(A^{\sharp}) \cong D_{qcoh}(A^{\dual},D_{A^{\dual}}).$$
\end{thm}

\begin{rmk}
In this paper we will only be interested in the case where $A = E$ is an elliptic curve. This special case is also covered by \cite[Thm. 4.10(2)]{Bezrukavnikov:fr}. 
\end{rmk}

As before we start by relating torsion sheaves on the surface $\Loc(E,1) = E^{\sharp}$ with local systems on $E^{\dual}$. Although the next Proposition is completely analogous to Lemma \ref{lemma:admissible}, it is more powerful, since every local system is admissible due to the fact that every vector bundle on a curve supporting a flat connection has degree zero.

\begin{prop}\label{prop:torsion-local}
Assuming that $k$ is an algebraically closed field of characteristic zero, the equivalence $\Phi_{CFT}$ of Theorem \ref{thm:laumon-rothstein} induces an equivalence of stacks $$\Tt(E^{\sharp},n) \cong \Loc(E^{\dual},n),$$ relating length $n$ torsion sheaves on the surface $E^{\sharp}$ to rank $n$ local systems on $E^{\dual}$.
\end{prop}

\begin{proof}
Let us denote by $S$ an affine scheme, $X$ an arbitrary smooth scheme and by $D_{qcoh}(X \times S,p_X^*D_X)$ the derived category of $p_X^*D_X$-modules, where $p_X: X \times S \to X$ is the canonical projection. Objects of this category should be thought of as $S$-families of complexes of $D_X$-modules. It is clear that we also have an equivalence of derived categories $$D_{qcoh}(A^{\sharp}\times S) \cong D_{qcoh}(A^{\dual}\times S,p_{A^{\dual}}^*D_{A^{\dual}}),$$ as it follows for instance from Proposition 4.1 in \cite{MR2669705} and the fact that the above equivalence of Laumon and Rothstein can be lifted to the canonical enhancements as stable $\infty$-categories.

Using the forgetful functor $$\Psi: D_{qcoh}(E^{\dual},D_{E^{\dual}}) \to D_{qcoh}(E^{\dual})$$ we can describe the underlying quasi-coherent sheaf $(\Psi \circ \Phi_{CFT})(\F)$ as the Fourier-Mukai transform $$D_{qcoh}(E^{\sharp}) \to D_{qcoh}(E^{\dual})$$ with integral kernel given by the universal flat connection $\Ll$ on $E^{\sharp} \times E$. As in the proof of Lemma \ref{lemma:admissible} we obtain therefore that $\Phi_{CFT}(\F)$ is a complex of a family of $D$-modules concentrated in a single degree.

Vice versa starting with a family of local systems $(V,\nabla)$ on $E^{\dual}$ we see from the existence of a composition series for $(V,\nabla)$ as in the proof of lemma \ref{lemma:admissible} that $\Phi_{CFT}^{-1}(V)$ is a torsion sheaf on $E^{\sharp}$.
\end{proof}

From a complex analytic viewpoint, Proposition \ref{prop:torsion-local} seems natural: since $\pi_1(E) \cong \mathbb{Z}^2$, the Riemann-Hilbert correspondence implies that $\Loc(E,n)$ is complex analytically isomorphic to the quotient stack $$[\{(A,B) \in \GL_n \times \GL_n|AB = BA\}/\GL_n].$$ This algebraic quotient stack in turn is equivalent to $\Tt(\mathbb{C}^{\times} \times \mathbb{C}^{\times},n)$, as we record below.

\begin{rmk}
Let $k$ be an algebraically closed field. There exists a canonical equivalence of stacks $$\Tt(\G_m^r \times \mathbb{A}^s,n) \cong [\{(A_1,\dots,A_r, A_{r+1},\dots, A_{r+s}) \in \GL_n^r \times \mathfrak{gl}_n^s|[A_i,A_j]=0 \; \forall\;(i,j)\}/\GL_n],$$ where $\GL_{n}$ acts by conjugation on this variety of matrices.
\end{rmk}

\begin{proof}
The data of a length $n$ torsion sheaf on $\G_m^r \times \mathbb{A}^s$ is equivalent to a rank $n$ $k$-vector space $V$, endowed with the structure of a $k[X_1^{\pm 1},\dots,X_r^{\pm 1},X_{r+1},\dots,X_{r+s}]$-module. This in turn is tautologically the same thing as a $k$-vector space $V$ together with $r+s$ pairwise commuting endomorphisms $(A_i)_{i=1,\dots r+s}$, such that $\det A_i \neq 0$ for $i \leq r$. As the same statements hold in families, we conclude the proof of the assertion.
\end{proof}

On the other hand, the surface $\mathbb{C}^{\times} \times \mathbb{C}^{\times}$ is complex analytically equivalent to $E^{\sharp}$, which induces an isomorphism of complex analytic stacks $$\Tt(E^{\sharp},n) \cong \Tt(\mathbb{C}^{\times} \times \mathbb{C}^{\times},n) \cong \Loc(E,n).$$

Proposition \ref{prop:torsion-local} allows us to prove a version of Theorem \ref{thm:toyexamples} for moduli spaces of parabolic local systems, by the exact same methods.

\begin{thm}\label{thm:flattoyexamples}
Let $k$ be an algebraically closed field of characteristic zero or $p > |\Gamma|$. The moduli space of stable parabolic local systems $\ssLoc(Q,{\lambda})$ of orbifold degree zero, associated to the orbifold $[E/\Gamma]$ with the same weights as in Theorem \ref{thm:toyexamples}, is naturally isomorphic to the $\Gamma$-Hilbert scheme of the surface $E^{\sharp}$.
\end{thm}

\subsection{Derived equivalences}

In Proposition \ref{prop:equivariant} we have shown that there is a derived equivalence $$D^b_{coh}([T^*E^{\dual}/\Gamma]) \cong D^b_{coh}([T^*E/\Gamma]).$$ Using Theorem \ref{thm:toyexamples} and the derived equivalence of Theorem \ref{thm:BKR} we arrive at a string of equivalences $$D^b_{coh}(\M) \cong D^b_{coh}([T^*E^{\dual}/\Gamma]) \cong D^b_{coh}([T^*E/\Gamma]) \cong D^b_{coh}(\M^{\dual}),$$ where $\M$ and $\M^{\dual}$ denote the respective moduli spaces of parabolic Higgs bundles. 

\begin{thm}\label{thm:toyautoduality}
Let $\M := \Higgs(\widehat{[E/\Gamma]},Q)$ denote the moduli space over the Hitchin base $\A$ studied in subsection \ref{ssec:crepant}.  We have a natural equivalence of derived categories of Fourier-Mukai type
$$\Phi: D^b_{coh}(\M) \cong D^b_{coh}(\M^{\dual}),$$
relative to $\A$,
extending the Fourier-Mukai transform for dual abelian varieties over the locus $\A^{\sm}$. The corresponding Fourier-Mukai kernel is given by a Cohen-Macaulay sheaf $\Pb$ on the fibre product $\M \times_{\A} \M^{\dual}$.
\end{thm}

\begin{proof}
This is an equivalence of Fourier-Mukai type relative to $\A$ by construction. Therefore we only need to verify the second assertion, namely that the integral kernel $\Pb$ restricts to the Fourier-Mukai transform associated to the Poincar\'e bundle $\Pp$ on $$\M^{\sm} \times_{\A^{\sm}} \M^{\dual,\sm}.$$ Over the smooth locus $\A^{\sm}$ the two morphisms
\begin{center}
\begin{tikzpicture}
\node at (0,1.5) {$\M$};
\node at (2,0) {$T^*E^{\dual}/\Gamma$};
\node at (4.4,1.5) {$[T^*E^{\dual}/\Gamma]$};

\draw [->,thick] (0.3,1.2) -- (1.6,0.3);
\draw [->,thick] (4.1,1.2) -- (2.3, 0.3);
\end{tikzpicture}
\end{center}
are actually isomorphisms and the restriction of the equivalence of Theorem \ref{thm:BKR} to the smooth locus (which is possible because of Lemma \ref{lemma:linear}) is the equivalence induced from this isomorphism. \'Etale locally on $\A^{\sm}$ we may identify the relative abelian variety given by the Hitchin fibration with $E^{\dual}$. We see that the equivalence in question is just Fourier-Mukai duality for the abelian variety $E$.

To verify the last assertion we need to show that the equivalence $D^b_{coh}(\M^{\dual}) \cong D^b_{coh}(\M)$ sends the $\M^{\dual}$-family of quasi-coherent sheaves on $\M^{\dual}$ given by the structure sheaf of the diagonal $\Delta_*\Oo_{\M^{\dual}}$ to a Cohen-Macaulay sheaf. This equivalence can be divided into several steps $$D^b_{coh}(\M) \cong D^b_{coh}([T^*E^{\dual}/\Gamma]) \cong D^b_{coh}([T^*E/\Gamma]) \cong D^b_{coh}(\M^{\dual}).$$ According to Theorem \ref{thm:toyexamples}, the composition of the first two equivalences send $\Delta_*\Oo_{\M^{\dual}}$ to the universal family $\Qb$ of Higgs orbibundles on $\M^{\dual} \times [T^*E/\Gamma]$. If $$\pi: [T^*E/\Gamma] \to [E/\Gamma]$$ denotes the canonical projection, we have that $$(\id_{\M} \times \pi)_*\Qb$$ the $\M$-family of vector bundles underlying the universal family of Higgs bundles $\Qb$. In particular, since $\pi: \supp \Qb \to [E/\Gamma]$ is finite, we see that $\Qb$ is Cohen-Macaulay. Therefore we need to show that the equivalence $\Psi: D^b_{coh}([T^*E/\Gamma]) \cong D^b_{coh}(\M^{\dual})$ sends $\Qb$ to a Cohen-Macaulay sheaf $\Pb$ on $\M \times_{\A} \M^{\dual}$. 
\begin{center}
\begin{tikzpicture}
\node at (1.5,1.5) {$\Z$};
\node at (0,0) {$[T^*E/\Gamma]$};
\node at (2.9,0) {$\M^{\dual}$};

\node at (0.5,0.9) {$p$};
\node at (2.4,0.9) {$q$};

\draw [->,thick] (1.3,1.3) -- (0.3,0.3);
\draw [->,thick] (1.7,1.3) -- (2.7,0.3);
\end{tikzpicture}
\end{center}
The universal $\Gamma$-cluster is endowed with a line bundle $\mathcal{K}$ and $\Psi$ can be written as $$Rq_*(Lp^*- \otimes^L \mathcal{K} ).$$ Because $q$ is a finite morphism and $\mathcal{K}$ is a line bundle, $$\Psi(\Qb) = R(\id \times q)_*(L(\id \times p)^*\Qb \otimes^L \mathcal{K})$$ is Cohen-Macaulay if and only if $Lp^*\Qb$ is Cohen-Macaulay. Lemma 2.3 of \cite{Arinkin:2010uq} implies Cohen-Macaulyness of this pullback, if $\M \times_{\A} [T^*E/\Gamma]$ is Gorenstein, $(\id \times p)$ is Tor-finite and $\M \times_{\A} \Z$ is Cohen-Macaulay. Tor-finiteness of $p$ follows from smoothness of $\M$ and is preserved by base change along a flat morphism. The two fibre products $\M \times_{\A} [T^*E/\Gamma]$ and $\M \times_{\A} \M$ are locally complete intersections (see tags 01UH, 01UI in \cite{stacks-project}), and $$ \M \times_{\A} \Z \to \M \times_{\A} \M$$ is a finite morphism, which implies Cohen-Macaulayness of $\M \times_{\A} \Z$.
\end{proof}

We obtain a similar result for moduli spaces of flat connections, which should be seen as an instance of the Geometric Langlands correspondence.

\begin{thm}
Let $\ssLoc(\widehat{[E/\Gamma]},Q,\lambda)$ denote the moduli space of local systems studied in subsection \ref{ssec:crepant}.  We have a natural equivalence of derived categories
$$\Phi_{GL}: D_{qcoh}(\ssLoc(\widehat{[E/\Gamma]},Q,\lambda)) \cong D_{qcoh}([E^{\dual}/\Gamma],D_{[E^{\dual}/\Gamma]}).$$
\end{thm}

\section{Hilbert schemes, Higgs bundles and local systems}

If $Q$ is a graph with a marked vertex $v$, we denote by $Q'$ the quiver obtained by adjoining an extra edge, linking $v$ with a new vertex $v'$. If $\lambda$ is a dimension vector for $Q$, we denote by $\lambda'$ the dimension vector satisfying $$\lambda'|Q = \lambda$$ and $\lambda'(v') = 1$.

If $Q$ is a Dynkin diagram, then the associated affine Dynkin diagram $\widetilde{Q}$ has a marked vertex $v$, called the affine vertex. In this section we discuss the geometric analogue for moduli spaces of Higgs bundles and local systems of the transition $$\widetilde{Q} \rightsquigarrow \widetilde{Q}'.$$

\subsection{Hilbert schemes as moduli spaces}

\begin{thm}\label{thm:Hilbert}
Let $k$ be an algebraically closed field of characteristic zero or $p > \max(|\Gamma|,n)$. We denote by $\M$ the moduli spaces of parabolic Higgs bundles $\ssHiggs(Q,{\lambda})$ from Theorem \ref{thm:toyexamples}. Then the Hilbert scheme $\M^{[n]}$ is again a moduli space of Higgs bundles. More precisely, we have $$\M^{[n]} \cong \ssHiggs(\widetilde{Q}',{(n\lambda)}'),$$ where the weights are $\alpha_i := \frac{i}{n}$ for $i < n$ and $1 > \alpha_n > \frac{n-1}{n}$; and the orbifold degree is zero. The Hitchin map $\M^{[n]} \to \A_n$ factors through the Hilbert-Chow map $$\M^{[n]} \to \M^{(n)} \to \A_1^{(n)} = \A_n,$$ where $\M^{(n)} \to \A_1^{(n)}$ is the map induced by $\M^n \to \A_1^n$.
\end{thm}

In the case of $Q = \widetilde{A}_0$ this is a theorem of Gorsky--Nekrasov--Rubtsov (\cite[Sect. 5.1]{MR1859601}).

\begin{proof}
Theorem \ref{thm:BKR} and Theorem \ref{thm:toyautoduality} imply that we have an equivalence $$D^b_{coh}(\M) \cong D^b_{coh}([T^*E/\Gamma]),$$ defined relative to $\A_1$. In particular we can show as in Lemma \ref{lemma:admissible} that the moduli stack of length $n$ torsion sheaves on $\M$ is equivalent to the moduli stack of $\Gamma$-equivariant rank $n$ admissible Higgs bundles on $E$: $$\Tt(\M) \cong \Higgs([E/\Gamma],n).$$  As in the proof of Theorem \ref{thm:toyexamples} we see that under this equivalence a morphism $\Oo_{\M} \to \F$ corresponds to a linear map $\tau_*(V^{\dual} \otimes \Oo/L_0^{-1}) \to k$, where $(V,\theta)$ is the $\Gamma$-equivariant Higgs bundle associated to $\F \in \Tt$. But $$\tau_*(V^{\dual} \otimes \Oo/L_0^{-1}) \cong E^{\dual}_0/F^{\dual}_1. $$ In particular we obtain a non-trivial linear map $$k \to F_{n_0-1}$$ by dualizing, i.e. a nonzero vector $v \in F_{n_0-1}$.

We claim that the condition that $\Oo_{\M} \to \F$ is surjective, equivalent is to the fact that $(V,\theta,v)$ does not contain any degree zero Higgs subbundles containing $v$.

Let us assume that $\Oo_{\M}$ is surjective. If $(V,\theta,v)$ contains a non-trivial degree zero Higgs subbundle, which contains $v$, then there is a smallest such Higgs subbundle $(V',\theta,v)$ of rank $k < n$. In particular its transform $\Gg$ gives rise to a commutative diagram
\begin{center}
\begin{tikzpicture}
\node at (0,1.5) {$\Oo_{\M}$};
\node at (1.52,0) {$\Gg.$};
\node at (3,1.5) {$\F$};

\draw [->,thick] (0.3,1.2) -- (1.2,0.3);
\draw [->,thick] (0.4,1.5) -- (2.7,1.5);
\draw [->,thick] (1.8,0.3) -- (2.7,1.2);
\end{tikzpicture}
\end{center}
Because the horizontal arrow is surjective and $W$ is a length $k$ torsion sheaf, this is a contradition.

Similarly one shows that if $(V,\theta,v)$ does not contain a non-trivial degree zero Higgs subbundle containing $v$, then the corresponding morphism $\Oo_{\M} \to \F$ is surjective. Namely, if it is not surjective, its image gives rise to a non-trivial Higgs subbundle of $(V,\theta,v)$ containing $v$. Stability is checked as in the proof of Theorem \ref{thm:toyexamples}.

We obtain a morphism of $\A_n$-spaces $$\M^{[n]} \to \ssHiggs(\widetilde{Q}',{(n\lambda)}'),$$ as the type of the corresponding parabolic bundle can be checked for a single point, by connectivity of the moduli spaces, for instance over the locus of smooth spectral curves. Properness of the Hitchin morphism and the fact that both spaces have equal dimension and are connected, imply that this morphism is surjective. In particular we may conclude that every parabolic Higgs bundle in $\ssHiggs(\widetilde{Q}',{(n\lambda)}')$ is admissible. This implies that the above morphism is an isomorphism, with the inverse given by the inverse Fourier-Mukai transform.
\end{proof}

There is an analogous statement for moduli spaces of local systems that is proved by the same means.

\begin{thm}\label{thm:localHilbert}
Let $k$ be an algebraically closed field of characteristic zero or $p > \max(|\Gamma|,n)$. We denote by $\M$ one of the moduli spaces of parabolic local systems $\ssLoc(Q,{\lambda})$ from Theorem \ref{thm:flattoyexamples} defined over an algebraically closed field of zero characteristic. Then the Hilbert scheme $\M^{[n]}$ is again a moduli space of local systems. More precisely, we have $$\M^{[n]} \cong \ssLoc(\widetilde{Q}',{(n\lambda)}'),$$ where the weights are as in Theorem \ref{thm:Hilbert}, and the orbifold degree is fixed to be zero. 
\end{thm}

\subsection{Derived equivalences}

We begin this subsection with the following observation:

\begin{lemma}\label{lemma:Hilbert}
Let $X$ and $Y$ be two quasi-projective smooth surfaces defined over an algebraically closed field of characteristic $p > n$. If we have an equivalence of derived categories $$D^b_{coh}(X) \cong D^b_{coh}(Y)$$ of Fourier-Mukai type then this induces an equivalence $$D^b_{coh}(X^{[n]}) \cong D^b_{coh}(Y^{[n]})$$ of Fourier-Mukai type.
\end{lemma}

\begin{proof} 
Lemma $\ref{lemma:products}$ allows us to take the $n$-th power of the equivalence $D^b_{coh}(X) \cong D^b_{coh}(Y)$, $$D^b_{coh}(X^n) \cong D^b_{coh}(Y^n).$$ On both spaces we have a natural action of the symmetric group $S_n$ by permuting the factors. The integral kernel is a sheaf naturally endowed with an $S_n$-equivariant structure; therefore we may apply Lemma \ref{lemma:equivariant} and conclude that $$D^b_{coh}([X^n/S_n]) \cong D^b_{coh}([Y^n/S_n]).$$ Together with Corollary \ref{cor:Hilbert} we obtain $$D^b_{coh}(X^{[n]}) \cong D^b_{coh}(Y^{[n]}).$$
\end{proof}

\begin{thm}\label{thm:autoduality}
We denote by $\M^{[n]}$ the moduli space of parabolic Higgs bundles associated to $[E/\Gamma]$ of Theorem \ref{thm:Hilbert}. By $\M^{\dual[n]}$ we denote the same moduli space for $\widehat{[E^{\dual}/\Gamma]}$. Both moduli spaces $\M^{[n]}$ and $\M^{\dual[n]}$ are $\A$-spaces, where $\A$ is the Hitchin base. Under the assumptions of Lemma \ref{lemma:Hilbert} there is a canonical equivalence of derived categories $$D^b_{coh}(\M^{[n]}) \cong D^b_{coh}(\M^{\dual[n]}),$$ relative to $\A$. The integral kernel of this derived equivalence is a Cohen-Macaulay sheaf $\Pb$ on $\M^{[n]} \times_{\A} \M^{\dual[n]}$, which restricts to the Poincar\'e bundle $\Pp$ over the locus of smooth spectral curves $\A^{\sm}$.
\end{thm}

\begin{proof}
This is a consequence of Theorem \ref{thm:Hilbert} and Lemma \ref{lemma:Hilbert}. Note that the construction in the proof of this lemma respects the morphism to the Hitchin base (due to Lemma \ref{lemma:linear}). The last two assertions are verified as in the proof of Theorem \ref{thm:toyautoduality}, with the single exception that this time the sheaf $\Qb$ cannot be thought of as a universal family of Higgs bundles. Therefore Cohen-Macaulayness has to be established by different means. The sheaf $\Qb$ is the transform of the structure sheaf $\Oo_{\Z}$ of the universal $S_n$-cluster on $\M^{[n]} \times_{\A} [M^n/S_n]$ along the equivalence $D^b_{qcoh}([M^n/S_n]) \cong D^b_{qcoh}([M^{\dual n}/S_n])$; we denote the integral kernel of the latter equivalence by $\mathcal{R}$, it is Cohen-Macaulay according to Theorem \ref{thm:toyautoduality} and Lemma \ref{lemma:products}. Let $\iota: \Z \to [M^n/S_n]$ be the canonical morphism; Lemma 2.3 of \cite{Arinkin:2010uq} implies that $L\iota^*\mathcal{R}$ is Cohen-Macaulay, since $\Z$ is finite over $\M^{[n]}$ and therefore Cohen-Macaulay itself. The natural morphism $\pi: \Z \to \M^[n]$ is finite, and so is every base change thereof. In particular we obtain that the transform of $\Oo_{\Z}$ is Cohen-Macaulay, as we wanted.
\end{proof}

Similarly one obtains an analogue for local systems.

\begin{thm}\label{thm:localduality}
We denote by $\M^{[n]}$ the moduli space of parabolic local systems associated to the weighted curve $\widehat{[E/\Gamma]}$ studied in Theorem \ref{thm:localHilbert}. Under the assumptions of Lemma \ref{lemma:Hilbert} there is a canonical equivalence of derived categories $$D_{qcoh}(\M^{[n]}) \cong D_{qcoh}([[E/\Gamma]^n/S_n],D_{[[E/\Gamma]^n/S_n]}).$$
\end{thm}

\subsection{Moduli of parabolic local systems in positive characteristic}

In a previous paper (\cite{Groechenig:2012uq}) the author investigated a general relation between the moduli stacks and spaces of local systems and Higgs bundles on a curve $X$ defined over an algebraically closed field $k$ of positive characteristic. Extending a result from Bezrukavnikov--Braverman (\cite{Bezrukavnikov:fr}), it is shown there that the two moduli stacks are \'etale locally equivalent over the Hitchin base. The Hitchin map for local systems exists only in positive characteristic and is constructed using the $p$-curvature (\cite{MR1810690}). 

Let $E$ be an elliptic curve, the moduli space of rank one and degree zero Higgs bundles on $E$ is given by $T^*E^{\dual}$; the moduli space of local systems $\Loc$ is an extension of $E^{\dual}$ by the vector space $\A = H^0(E,\Omega^1_E)$. The Hitchin map for local systems is a map $$\HdR: \Loc \to \A^{(1)},$$ where $\A^{(1)}$ denotes the Frobenius twist of the variety $\A$. Let $\omega \in H^0(E,\Omega_E^1)$; formula 2.1.16 in \cite{MR565469} asserts that $$\tau := \HdR(\Oo_E,d+\omega): \A \to \A^{(1)}$$ is a sum of a $p$-linear and a linear map. Since $\A$ is a one-dimensional vector space, we may assume without loss of generality that this morphism is given by the Artin-Schreier-morphism $$\mathrm{AS}: \mathbb{A}^1 \to \mathbb{A}^1,$$ which sends $\lambda \to \lambda^p - \lambda$. In particular this morphism is \'etale. This is an explicit construction of an \'etale local section of the Hitchin morphism $\Loc \to \A^{(1)}$, i.e. \'etale locally around every $a \in \A^{(1)}$ we assign a solution to the equation $a = \HdR(\Oo,d + \omega)$. Any other local system $(\mathcal{E},\nabla)$ over $a$ can now be tensored with $(\Oo,d+\omega)^{\dual}$ to obtain $(\mathcal{E},\nabla')$; which is a flat connection of $p$-curvature zero. According to a theorem of Cartier such local systems are in bijection with line bundles on the Frobenius twist $E^{(1)}$ (\cite[Thm. 5.1.1]{Katz:fk}). We conclude that after base change along $\tau$ we obtain a natural isomorphism\footnote{The author thanks C. Pauly for explaining this example to him.} $$\Loc \times_{\A^{(1)}} \A \cong (E^{(1),\dual} \times \A^{(1)}) \times_{\A^{(1)}} \A.$$ 

One would expect that this \'etale local equivalence induces a similar comparison result for $\Gamma$-Hilbert schemes of the cotangent bundle of $E^{(1)}$ and $\Loc(E)$ respectively. Theorems \ref{thm:toyexamples} and \ref{thm:Hilbert} suggest that a similar \'etale local equivalence holds as well for certain moduli spaces of parabolic local systems. In the paper \cite{MR2577203} Nevins establishes this local equivalence on the locus of regular spectral curves (in the mirabolic case) in order to generalize Bezrukavnikov--Braverman's work on the Geometric Langlands correspondence in positive characteristic to the mirabolic case. 

Strictly speaking, the parabolic case is not covered by the authors paper \cite{Groechenig:2012uq}, as we assume there that $X$ is a curve. Nonetheless, the same methods used there to deduce the \'etale local equivalence of local systems and Higgs bundles, apply to orbicurves as well. 

\begin{prop}
Let $X$ be a smooth complete orbicurve defined over an algebraically closed field $k$ of characteristic $p$, satisfying assumption \ref{ass}. Then the moduli stack of rank $n$ local systems $\Loc(X,n)$ is \'etale locally equivalent to $\Higgs(X^{(1)},n)$ relative to the Hitchin base $\A^{(1)}$. The same assertion holds for (semi)stable local systems and Higgs bundles.
\end{prop}

\begin{proof}
The only part of the proof that is sensitive to orbifold structures is Theorem 3.4 in \cite{Groechenig:2012uq}. But as before it suffices to show that every $\G_m$-gerbe neutralizes on a smooth complete orbicurve $Y$ defined over an algebraically closed field. After this is established one can evoke the same argument as in \emph{loc. cit.} to deduce representability of the stack of splittings; the corresponding result for the Picard stack is proved in \cite[Thm. 5.1]{Aoki:fk}. As in the curve case one expects to be able to deduce this from Tsen's Theorem (\cite[Ex. III.2.22 (d)]{MR559531}), which states that for a smooth curve defined over an algebraically closed field $k$, every $\G_m$-gerbe neutralizes over the generic point $$H^2_{et}(\Spec K(X),\G_m).$$ It turns out that the argument of \emph{loc. cit.} applies to our situation, after some small modifications. Let $U \subset X$ be a maximal schematic open subset of $X$; we have $K(X) = K(U)$ due to the birational nature of $K(X)$. We denote by $\mathcal{K}^{\times}_X$ the sheaf of non-vanishing rational function on $X$; the sheaf-theoretic quotient $\mathcal{K}^{\times}_X/\Oo^{\times}_X$ is the sheaf of divisors $\mathrm{Div}_X$. By definition we have a short exact sequence $$1 \to \Oo_X^{\times} \to \mathcal{K}_X^{\times} \to \mathrm{Div}_X \to 1,$$ called \emph{Weil divisor sequence}. Taking global sections we obtain the following interesting bit of the associated long exact sequence 
$$H^1_{et}(X,\mathrm{Div}_X) \to H^2_{et}(X,\Oo_X^{\times}) \to H^2_{et}(X,\mathcal{K}^{\times}_X).$$
The complement $X - U$ is a union of finitely many orbifold points $p_1,\dots p_k$, with stabilizer group $\Gamma_i$. We then have $$\mathrm{Div}_X = \bigoplus_{x \in U(k)} i_{x,*}\mathbb{Z} \oplus \bigoplus_{i=1}^k j_{i,*}\mathbb{Z},$$ where $i_x$ denotes the closed immersion $\Spec k \to X$ associated to a point $x \in U$, and $j_i$ denotes the closed immersion $$B\Gamma_i \to X$$ associated to an orbifold point $p_i$. Since closed immersions are finite, we have $R^li_{x,*}= 0$ and $R^lj_{i,*}= 0$ for $l > 0$. In particular we obtain the vanishing result $H^1_{et}(X,\mathrm{Div}_X) = 0$.

Let $\eta: \Spec K(X) \to X$ denote the inclusion of the generic point. We have $$\mathcal{K}_X^{\times} = \eta_*\G_m$$ and therefore have to show that $R^l\eta_*\G_m = 0$ for $l > 0$. As in \emph{loc. cit.} this is checked stalkwise, by identifying $(R^l\eta_*\G_m)_{\bar{x}}$ for every $x$ different from the generic point with the Galois cohomology group $$H^l(\Spec K_x,\G_m) = 0,$$ where $K_x$ denotes the fraction field of the Henselization of $x$ and evoking a vanishing result in Galois cohomology due to Lang.
\end{proof}

\bibliographystyle{myamsalpha}
\bibliography{master}

E-mail: groechenig@maths.ox.ac.uk

\end{document}